\documentclass[a4paper,10pt]{amsart}
\usepackage[top=0.8in, bottom=0.8in, left=1.25in, right=1.25in]{geometry}
\usepackage{amsmath}
\usepackage{amsfonts}
\usepackage{amssymb}
\usepackage{graphicx}
\usepackage{mathrsfs}
\usepackage{pb-diagram}
\usepackage{epstopdf}
\usepackage{CJK}
\usepackage{slashed}

\usepackage{amscd,amsthm,curves,enumerate,latexsym,multibox}
\usepackage{xypic}
\usepackage[all]{xy}
\usepackage{epstopdf}
\newtheorem{theorem}{Theorem}[section]
\newtheorem{lemma}[theorem]{Lemma}
\newtheorem{proposition}[theorem]{Proposition}

\newtheorem{conjecture}[theorem]{Conjecture}
\newtheorem{theorem/definition}[theorem]{Theorem/Definition}
\theoremstyle{definition}
\newtheorem{definition}[theorem]{Definition}
\newtheorem{remark}[theorem]{Remark}
\newtheorem{example}[theorem]{Example}

\newtheorem{question}[theorem]{Question}

\begin{document}
\title[Relative DT theory for Calabi-Yau 4-folds]{Relative Donaldson-Thomas theory \\ for Calabi-Yau 4-folds}
\author{Yalong Cao}
\address{The Institute of Mathematical Sciences and Department of Mathematics, The Chinese University of Hong Kong, Shatin, Hong Kong}
\email{ylcao@math.cuhk.edu.hk}

\author{Naichung Conan Leung}
\address{The Institute of Mathematical Sciences and Department of Mathematics, The Chinese University of Hong Kong, Shatin, Hong Kong}
\email{leung@math.cuhk.edu.hk}

\maketitle
\begin{abstract}
Given a complex 4-fold $X$ with an (Calabi-Yau 3-fold) anti-canonical divisor $Y$, we study relative Donaldson-Thomas invariants for this pair, which are elements in the Donaldson-Thomas cohomologies of $Y$. We also discuss gluing formulas which relate relative invariants and
$DT_{4}$ invariants for Calabi-Yau 4-folds.
\end{abstract}

%\tableofcontents

\section{Introduction}
Donaldson-Thomas invariants ($DT_{3}$ invariants for short) were proposed by Donaldson and Thomas \cite{dt}, and defined in Thomas' thesis \cite{th}.
They count stable sheaves on Calabi-Yau 3-folds, which are related to many other interesting subjects, including
Gopakumar-Vafa conjecture on BPS numbers in string theory \cite{gv}, \cite{hosono}, \cite{kl} and
MNOP conjecture \cite{mnop}, \cite{mnop2}, \cite{moop}, \cite{pandpixton} relating $DT_{3}$ invariants to Gromov-Witten invariants.
The generalization of $DT_{3}$ invariants to count strictly semi-stable sheaves is due to Joyce and Song \cite{js} using Behrend's result \cite{behrend}.

Kontsevich and Soibelman proposed generalized as well as motivic $DT$ theory for Calabi-Yau 3-categories \cite{ks}, which was later studied by Behrend-Bryan-Szendr\"{o}i \cite{bbs} for Hilbert schemes of points. The wall-crossing formula \cite{ks}, \cite{js} is an important structure for
Bridgeland's stability condition \cite{bridgeland} and Pandharipande-Thomas invariants \cite{pt}, \cite{toda}.

As a categorification of Donaldson-Thomas invariants, Brav, Bussi, Dupont, Joyce and Szendroi \cite{bbdjs} and Kiem and Li \cite{kl}
recently defined a cohomology theory for Calabi-Yau 3-folds whose Euler characteristic is the $DT_{3}$ invariant.
The point is that moduli spaces of simple sheaves on Calabi-Yau 3-folds are critical points of holomorphic functions locally \cite{bbj}, \cite{js},
and we could consider perverse sheaves of vanishing cycles of these functions. They glued these local perverse sheaves and defined
its hypercohomology as $DT_{3}$ cohomology. In general, such a gluing requires a square root of the determinant
line bundle of the moduli space \cite{hua1}, \cite{ks}, \cite{no}.

As an extension of Donaldson-Thomas invariants to Calabi-Yau 4-folds, Borisov and Joyce \cite{bj} and the authors \cite{cao}, \cite{caoleung}
developed $DT_{4}$ invariants (or 'holomorphic Donaldson invariants') which count stable sheaves on Calabi-Yau 4-folds.
It is desirable to construct a TQFT type structure for these $DT_{4}$ and $DT_{3}$ theories.
The purpose of this paper is to make some initial steps in this direction.
We remark that Joyce also has a program of establishing TQFT structures on
Calabi-Yau 3 and 4-folds \cite{joyce} using Pantev-T\"{o}en-Vaqui\'{e}-Vezzosi's shifted symplectic structures on derived schemes \cite{ptvv}. \\

Our set-up is a smooth Calabi-Yau 3-fold $Y=s^{-1}(0)$ as an anti-canonical divisor of a complex projective 4-fold $X$, where
$s\in \Gamma(X,K_{X}^{-1})$. Then $1/s$ is nowhere vanishing inside $X\backslash Y$ which gives a trivialization of its
canonical bundle. Thus $X\backslash Y$ is an open Calabi-Yau 4-fold which has a compactification $X$ by adding a compact Calabi-Yau 3-fold.

We consider any Gieseker moduli space $\mathfrak{M}_{X}$ of semi-stable sheaves which consists of slope-stable bundles only, and assume
there exists a restriction morphism
$r: \mathfrak{M}_{X}\rightarrow \mathfrak{M}_{Y}$ to a Gieseker moduli space of stable sheaves on $Y$ (see Theorem \ref{restriction thm} for its existence). The deformation-obstruction theory
associated to $r$ is described as follows: for any stable bundle $E\in\mathfrak{M}_{X}$, we have an exact sequence
\begin{equation}0\rightarrow H^{1}(X,End_{0}E\otimes K_{X})\rightarrow H^{1}(X,End_{0}E)\rightarrow H^{1}(Y,End_{0}E|_{Y})
\rightarrow \nonumber \end{equation}
\begin{equation}\rightarrow H^{2}(X,End_{0}E\otimes K_{X})\rightarrow H^{2}(X,End_{0}E)\rightarrow H^{2}(Y,End_{0}E|_{Y})  \rightarrow
\nonumber \end{equation}
\begin{equation}\rightarrow H^{3}(X,End_{0}E\otimes K_{X})\rightarrow
H^{3}(X,End_{0}E)\rightarrow 0. \quad \quad \quad \quad \quad \quad \quad \quad \nonumber \end{equation}
Note that the transpose of the above sequence with respect to Serre duality pairings on $X$ and $Y$ is itself \cite{dt}. This is the
key property for the definition of relative $DT_{4}$ virtual cycles, which we define for the following three good cases.

\textbf{Case I}. [Rigid case] If every $E\in \mathfrak{M}_{X}$ satisfies $H^{1}(Y,End_{0}E|_{Y})=0$, then the above long exact sequence
breaks into canonical isomorphisms
\begin{equation}H^{1}(X,End_{0}E) \cong H^{3}(X,End_{0}E)^{*}, \quad H^{2}(X,End_{0}E) \cong H^{2}(X,End_{0}E)^{*}. \nonumber \end{equation}
Similar to the case of Calabi-Yau 4-folds \cite{bj}, $\mathfrak{M}_{X}$ will have a $(-2)$-shifted symplectic structure in the sense of PTVV \cite{ptvv}.
By Borisov-Joyce \cite{bj}, there exists a virtual cycle, which we define to be the relative $DT_{4}$ virtual cycle.

\textbf{Case II}. [Surjective case] If $r: \mathfrak{M}_{X}\rightarrow\mathfrak{M}_{Y}$ is a surjective map between smooth moduli spaces (throughout this paper, unless
specified otherwise, smooth moduli spaces mean their Kuranishi maps are zero), we obtain a canonical isomorphism
\begin{equation}H^{2}(X,End_{0}E)\cong H^{2}(X,End_{0}E)^{*}, \nonumber \end{equation}
which endows the obstruction bundle $Ob_{\mathfrak{M}_{X}}$ with a non-degenerate quadratic form. Then the relative $DT_{4}$ virtual cycle
is defined to be the Euler class of the self-dual subbundle of $Ob_{\mathfrak{M}_{X}}$ as in Definition 5.12 \cite{caoleung}.

\textbf{Case III}. [Injective case] If $r: \mathfrak{M}_{X}\rightarrow\mathfrak{M}_{Y}$ is an injective map between smooth moduli spaces, we obtain an exact sequence
\begin{equation}0 \rightarrow H^{1}(X,End_{0}E)\rightarrow H^{1}(Y,End_{0}E|_{Y})
\rightarrow H^{2}(X,End_{0}E\otimes K_{X}) \nonumber \end{equation}
\begin{equation}\rightarrow  H^{2}(X,End_{0}E) \rightarrow H^{2}(Y,End_{0}E|_{Y}) \rightarrow H^{3}(X,End_{0}E\otimes K_{X}) \rightarrow 0.
\nonumber \end{equation}
This determines a surjective map
\begin{equation}s: Ob_{\mathfrak{M}_{X}}\twoheadrightarrow \mathcal{N}^{*}_{\mathfrak{M}_{X}/\mathfrak{M}_{Y}} \nonumber \end{equation}
from the obstruction bundle of $\mathfrak{M}_{X}$ to the conormal bundle of $\mathfrak{M}_{X}$ inside $\mathfrak{M}_{Y}$,
and a non-degenerate quadratic form on the reduced bundle $Ob_{\mathfrak{M}_{X}}^{red}\triangleq Ker(s)$.
As in \textbf{Case II}, we define the relative $DT_{4}$ virtual cycle
$[\mathfrak{M}_{X}^{rel}]^{vir}\in H_{*}(\mathfrak{M}_{X},\mathbb{Z})$ to be the Euler class
of the self-dual subbundle of $Ob^{red}_{\mathfrak{M}_{X}}$. Note that when $\mathfrak{M}_{X}$ is smooth, $r$ is injective and
a neighbourhood of $r(\mathfrak{M}_{X})\subseteq \mathfrak{M}_{Y}$ is smooth, $[\mathfrak{M}_{X}^{rel}]^{vir}$ can also be defined in a similar way.
It is easy to check these definitions of relative $DT_{4}$ virtual cycles in \textbf{Cases I-III} are all compatible.
We compute examples for relative $DT_{4}$ virtual cycles in Proposition \ref{lq example compute}.

To sum up, we define the notion of admissibility for Gieseker moduli spaces.
\begin{definition}\label{admissible}
Let $Y$ be a smooth anti-canonical divisor of a projective 4-fold $X$, and $\mathfrak{M}_{X}$ be a Gieseker moduli space of semi-stable sheaves.
$\mathfrak{M}_{X}$ is admissible with respect to $(X,Y)$ if \\
(i) $\mathfrak{M}_{X}$ consists of slope-stable bundles only, and \\
(ii) there exists a restriction morphism $r: \mathfrak{M}_{X}\rightarrow \mathfrak{M}_{Y}$ to a Gieseker moduli space of stable sheaves on $Y$ \footnote[1]{Theorem \ref{restriction thm} ensures we have many such examples.}.
\end{definition}
\begin{theorem/definition}
If $\mathfrak{M}_{X}$ is admissible with respect to $(X,Y)$, then the relative $DT_{4}$ virtual cycle exists, i.e.
\begin{equation}[\mathfrak{M}_{X}^{rel}]^{vir}\in H_{*}(\mathfrak{M}_{X},\mathbb{Z}_{2})\nonumber \end{equation}
provided that any one of the following conditions holds, \\
%consists of slope-stable bundles only and has a well-defined restriction morphism (see Theorem \ref{restriction thm} for its existence)
%\begin{equation}r: \mathfrak{M}_{X}\rightarrow \mathfrak{M}_{Y} \nonumber \end{equation}
%to a Gieseker moduli space of stable sheaves on $Y$. If we assume \\
(1) $r(\mathfrak{M}_{X})$ is rigid, i.e. $H^{1}(Y,End_{0}E|_{Y})=0$ for any $E\in \mathfrak{M}_{X}$; or \\
(2) $r$ is surjective between smooth moduli spaces; or \\
(3) $r$ is injective between smooth moduli spaces
(at least when restricted to a neighbourhood of $r(\mathfrak{M}_{X})$ in $\mathfrak{M}_{Y}$).

Furthermore, $[\mathfrak{M}_{X}^{rel}]^{vir}$ will be defined over integer if $r: \mathfrak{M}_{X}\rightarrow \mathfrak{M}_{Y}$ has a relative orientation (Definition \ref{def on rel ori 1}) \footnote[2]{See Theorem \ref{thm 1} and Proposition \ref{pro on rel ori 1} for some partial verification of the existence of relative orientations.}.
\end{theorem/definition}

In general, the virtual dimension of $[\mathfrak{M}_{X}^{rel}]^{vir}$ is not zero, and we introduce the $\mu$-map to cut it down and define the relative $DT_{4}$ invariant. The relative $DT_{4}$ invariant is a map
\begin{equation}v(\mathfrak{M}_{X}): Sym^{*}\big(H_{*}(X,\mathbb{Z}) \otimes \mathbb{Z}[x_{1},x_{2},...]\big)
\rightarrow \mathbb{H}^{*}(\mathfrak{M}_{Y},\mathcal{P}^{\bullet}_{\mathfrak{M}_{Y}}), \nonumber \end{equation}
where $\mathcal{P}^{\bullet}_{\mathfrak{M}_{Y}}$ is the perverse sheaf constructed
by Brav-Bussi-Dupont-Joyce-Szendroi \cite{bbdjs} and Kiem-Li \cite{kl}. In \textbf{Cases I-III}, if $\mathfrak{M}_{Y}$ is smooth, the perverse sheaf
$\mathcal{P}^{\bullet}_{\mathfrak{M}_{Y}}$ is the $\mathbb{C}$-constant sheaf (up to some degree shift) and the relative $DT_{4}$ invariant $v(\mathfrak{M}_{X})$
is defined by pairing the relative $DT_{4}$ virtual cycle, $\mu$-map and
pull-back classes from $H^{*}(\mathfrak{M}_{Y},\mathbb{C})$.  \\

So far, we only work with holomorphic bundles (i.e. $\mathfrak{M}_{X}$ consists of bundles only).
To extend to other coherent sheaves, say ideal sheaves of subschemes, one difficulty
is that we will not have well-defined restriction maps in the usual sense as subschemes would sit inside the divisor $Y\subseteq X$. To handle this issue, we introduce Li-Wu's good degenerations of Hilbert schemes \cite{liwu}. Li-Wu's idea is to blow up the divisor once subschemes sit inside $Y$
(like neck-stretching  in Donaldson theory) and we are then reduced to consider subschemes which are 'transversal' to $Y$.

Working with their good degenerations, we study extensions of the above \textbf{Cases I-III} to ideal sheaves cases. In particular, we study the
obstruction theory of moduli spaces of relative ideal sheaves (Lemma \ref{fund exact seq for dege }) and discuss a gluing formula (Theorem \ref{dege formula}) based on certain conjectures.
We also compute examples on relative $DT_{4}$ virtual cycles (Example \ref{ideal sheaves of one pt}$-$Example \ref{ex rel DT4/DT3}) based on Definition \ref{rel DT4 vc for ideal sheaves}, which include
\begin{example}(Generic quintic in $\mathbb{P}^{4}$, Example \ref{ex quintic})

We take $X=\mathbb{P}^{4}$ which contains a generic quintic 3-fold $Y=Q$ as its anti-canonical divisor, and consider the primitive curve class $[H]\in H_{2}(X,\mathbb{Z})$. Ideal sheaves of curves representing this class have Chern character
$c=(1,0,0,-PD([H]),\frac{3}{2})$ and we denote their moduli space by $I_{\frac{3}{2}}(X,[H])$ ($\cong Gr(2,5)$).
The generic quintic $Q$ contains $2875$ rigid degree $1$ rational curves and $I_{\frac{3}{2}}(X,[H])$ contains a finite subset $S$ with $2875$ points. $I_{\frac{3}{2}}(X,[H])\backslash S$ has a well-defined restriction morphism to $Hilb^{5}(Q)$.
To extend the morphism across those $2875$ points, we introduce Li-Wu's expanded pair $X[1]_{0}=X\cup\Delta_{1}$, $Y[1]_{0}(\cong Q)\subseteq\Delta_{1}$, where $\Delta_{1}\cong \mathbb{P}(\mathcal{O}_{Q}\oplus
\mathcal{O}_{\mathbb{P}^{4}}(5)|_{Q})$, and consider the moduli space $I_{\frac{3}{2}}(X[1]_{0},[H])$ of
relative ideal sheaves of curves. Geometrically, it
is the blow up of $I_{\frac{3}{2}}(X,[H])$ along those $2875$ points, i.e.
\begin{equation}I_{\frac{3}{2}}(X[1]_{0},[H])\cong Bl_{S}(Gr(2,5)),  \nonumber \end{equation}
where each exceptional divisor corresponds to a $Hilb^{5}(\mathbb{P}^{1})$ for each $\mathbb{P}^{1}\subseteq Q$.
We then have a injective restriction morphism
\begin{equation}I_{\frac{3}{2}}(X[1]_{0},[H])\rightarrow Hilb^{5}(Y[1]_{0}),  \nonumber \end{equation}
\begin{equation}I_{C}\mapsto I_{C}|_{Y[1]_{0}}  \nonumber \end{equation}
with smooth image. Conditions in Definition \ref{rel DT4 vc for ideal sheaves} (or \textbf{Case III}) are satisfied and the relative $DT_{4}$ virtual cycle is the usual fundamental class of the moduli space $I_{\frac{3}{2}}(X[1]_{0},[H])\cong Bl_{S}(Gr(2,5))$.
\end{example}
\begin{proposition}
For $(\mathbb{P}^{4},Q)$, we have a restriction morphism
\begin{equation}I_{\frac{3}{2}}(\mathbb{P}^{4}[1]_{0},[H])\rightarrow Hilb^{5}(Q)  \nonumber \end{equation}
from the moduli space of relative ideal sheaves of degree 1 rational curves in $\mathbb{P}^{4}$ to the Hilbert scheme of five points in a generic quintic 3-fold. The relative $DT_{4}$ virtual cycle of $I_{\frac{3}{2}}(\mathbb{P}^{4}[1]_{0},[H])$ is the usual fundamental class of the moduli space $I_{\frac{3}{2}}(\mathbb{P}^{4}[1]_{0},[H])\cong Bl_{S}(Gr(2,5))$.
\end{proposition}
We adapt Li-Wu's good degenerations to torsion sheaves and verify Conjectures \ref{conj0}, \ref{conj1} in the following case.
\begin{example}(Relative $DT_{4}$/$DT_{3}$, Example \ref{ex rel DT4/DT3})

Let $X=Y_{1}\times \mathbb{P}^{1}$ which contains $Y=(Y_{1}\times 0)\sqcup (Y_{1}\times \infty)$ as an anti-canonical divisor,
where $Y_{1}$ is a compact Calabi-Yau 3-fold.
%$H^{*}(X,\mathbb{Z})\cong H_{2}(Q,\mathbb{Z})\oplus H_{2}(\mathbb{P}^{1},\mathbb{Z})\cong \mathbb{Z}\times\mathbb{Z}$.
We denote $\mathfrak{M}_{c}(Y_{1})$ to be a Gieseker moduli space of torsion-free semi-stable sheaves on $Y_{1}$ with
Chern character $c\in H^{even}(Y_{1},\mathbb{Q})$ (we assume there is no strictly semi-stable sheaf), and denote $\mathfrak{M}_{c}(X)$ to
be the moduli space of sheaves on $X$ which are push-forward of stable sheaves in $\mathfrak{M}_{c}(Y_{1}\times t)$ for some $t$
($\mathfrak{M}_{c}(X)\cong \mathfrak{M}_{c}(Y_{1})\times \mathbb{P}^{1}$).

To have a well-defined restriction map, we introduce $X[1]_{0}=\Delta_{-1}\cup X\cup\Delta_{1}$,
where $\Delta_{\pm1}\cong Y_{1}\times\mathbb{P}^{1}$, and consider the relative moduli space $\mathfrak{M}_{c}(X[1]_{0})$ with
\begin{equation}\mathfrak{M}_{c}(X[1]_{0})\cong \mathfrak{M}_{c}(Y_{1})\times\mathbb{P}^{1}.  \nonumber \end{equation}
Conjectures \ref{conj0}, \ref{conj1} hold and the relative $DT_{4}$ virtual cycle satisfies
\begin{equation}[\mathfrak{M}^{rel}_{c}(X[1]_{0})]^{vir}=DT_{3}(\mathfrak{M}_{c}(Y_{1}))\cdot[\mathbb{P}^{1}]\in
H_{2}(\mathfrak{M}_{c}(X[1]_{0}),\mathbb{Z}),\nonumber \end{equation}
where $DT_{3}(\mathfrak{M}_{c}(Y_{1}))$ is the Donaldson-Thomas invariant defined by Thomas \cite{th}.
\end{example}
\begin{theorem}
For $(X=Y_{1}\times \mathbb{P}^{1},Y=Y_{1}\times\{0,\infty\})$, where $Y_{1}$ is a compact Calabi-Yau 3-fold, we have a restriction map
\begin{equation}\mathfrak{M}_{c}(X[1]_{0})\cong \mathfrak{M}_{c}(Y_{1})\times\mathbb{P}^{1}\rightarrow pt,  \nonumber \end{equation}
from the moduli space of relative torsion sheaves coming from push-forward of stable sheaves in $\mathfrak{M}_{c}(Y_{1}\times t)$, where $\mathfrak{M}_{c}(Y_{1})$ is a Gieseker moduli space of torsion-free semi-stable sheaves on $Y_{1}$ consisting of no strictly semi-stable sheaf.

The relative $DT_{4}$ virtual cycle exists and satisfies
\begin{equation}[\mathfrak{M}^{rel}_{c}(X[1]_{0})]^{vir}=DT_{3}(\mathfrak{M}_{c}(Y_{1}))\cdot[\mathbb{P}^{1}]\in
H_{2}(\mathfrak{M}_{c}(X[1]_{0}),\mathbb{Z}),\nonumber \end{equation}
where $DT_{3}(\mathfrak{M}_{c}(Y_{1}))$ is the $DT_{3}$ invariant of $Y_{1}$ ($CY_{3}$) with respect to Chern character $c$.
\end{theorem}
Finally, we give a coherent description of the orientability issues involved in \textbf{Cases I-III} (Proposition \ref{prop on relative ori}) and
summarize them into the following definition.
\begin{definition}(Definition \ref{def of rel ori}) \label{def on rel ori 1}
Let $X$ be a smooth projective $4$-fold with a smooth anti-canonical divisor $Y\in|K^{-1}_{X}|$, and $r: \mathcal{M}_{X}\rightarrow\mathcal{M}_{Y}$
%\begin{equation}r: \mathcal{M}_{X}\rightarrow\mathcal{M}_{Y}  \nonumber \end{equation}
be a well-defined restriction morphism between coarse moduli spaces of simple sheaves on $X$ and $Y$ with fixed Chern classes respectively.
In this case, there exists a canonical isomorphism
\begin{equation}\alpha:(\mathcal{L}_{\mathcal{M}_{X}})^{\otimes2}\cong r^{*}\mathcal{L}_{\mathcal{M}_{Y}}.   \nonumber \end{equation}

A \emph{relative orientation} for morphism $r$ consists of a square root $(\mathcal{L}_{\mathcal{M}_{Y}}|_{\mathcal{M}^{red}_{Y}})^{\frac{1}{2}}$
of the determinant line bundle $\mathcal{L}_{\mathcal{M}_{Y}}|_{\mathcal{M}^{red}_{Y}}$ and an isomorphism
\begin{equation}\theta:\mathcal{L}_{\mathcal{M}_{X}}|_{\mathcal{M}^{red}_{X}} \cong
r^{*}(\mathcal{L}_{\mathcal{M}_{Y}}|_{\mathcal{M}^{red}_{Y}})^{\frac{1}{2}}  \nonumber \end{equation}
such that $\theta\otimes\theta\cong \alpha$ holds over $\mathcal{M}^{red}_{X}$ for the isomorphism $\alpha$.
\end{definition}
We then give the following partial verification of the existence of orientations.
\begin{theorem}(Weak relative orientability, Theorem \ref{thm on ori}) \label{thm 1} ${}$ \\
Let $Y$ be a smooth anti-canonical divisor in a projective $4$-fold $X$ with $Tor(H_{*}(X,\mathbb{Z}))=0$, $E\rightarrow X$ be a
complex vector bundle with structure group $SU(N)$, where $N\gg0$.
Let $\mathcal{M}_{X}$ be a coarse moduli scheme of simple holomorphic structures on $E$, which has a well-defined restriction morphism
\begin{equation}r: \mathcal{M}_{X}\rightarrow \mathcal{M}_{Y}, \nonumber \end{equation}
to a proper coarse moduli scheme of simple bundles on $Y$ with fixed Chern classes.

Then there exists a square root $(\mathcal{L}_{\mathcal{M}_{Y}}|_{\mathcal{M}^{red}_{Y}})^{\frac{1}{2}}$ of
$\mathcal{L}_{\mathcal{M}_{Y}}|_{\mathcal{M}^{red}_{Y}}$ such that
\begin{equation}c_{1}(\mathcal{L}_{\mathcal{M}_{X}}|_{\mathcal{M}^{red}_{X}})
=r^{*}c_{1}((\mathcal{L}_{\mathcal{M}_{Y}}|_{\mathcal{M}^{red}_{Y}})^{\frac{1}{2}}),
\nonumber \end{equation}
%\begin{equation}\mathcal{L}_{\mathcal{M}_{X}}\cong r^{*}(\mathcal{L}_{\mathcal{M}_{Y}}^{\frac{1}{2}}), \nonumber \end{equation}
where $\mathcal{L}_{\mathcal{M}_{X}}$ (resp. $\mathcal{L}_{\mathcal{M}_{Y}}$) is the determinant line bundle of $\mathcal{M}_{X}$
(resp. $\mathcal{M}_{Y}$).
\end{theorem}
Another partial verification is given as follows.
\begin{proposition}(Proposition \ref{prop on rel ori})\label{pro on rel ori 1} ${}$ \\
We assume $H^{1}(\mathcal{M}_{X},\mathbb{Z}_{2})=0$. Then relative orientations
for restriction morphism $r: \mathcal{M}_{X}\rightarrow\mathcal{M}_{Y}$ exist.
\end{proposition}
${}$ \\
\textbf{Acknowledgement}:
The first author is very grateful to Baosen Wu for many helpful discussions and
Dominic Joyce for introducing him his program of
studying TQFT structures on Calabi-Yau 3 and 4-folds during a visit at the Simons Center for Geometry and Physics on May 2014.
He also expresses his deep gratitude to Simon Donaldson for useful discussions and inviting him to give a talk on $DT_{4}$ theory
in the G2 manifolds workshop at the Simons Center on Sept. 2014, where part of this work was done. We thank Yan Soibelman for comments on the appropriate citation of references. The work of the second author was substantially supported by grants from the Research Grants Council of the Hong Kong Special
Administrative Region, China (Project No. CUHK401411 and CUHK14302714).

\section{Review of basic facts in DT theory}

\subsection{Some basic facts in $DT_{4}$ theory}
We start with a complex projective Calabi-Yau 4-fold $(X,\mathcal{O}_{X}(1))$ ($Hol(X)=SU(4)$) with a Ricci-flat K\"ahler metric $g$ \cite{yau}, a K\"ahler form $\omega$, a holomorphic four-form $\Omega$, and a topological bundle with a Hermitian metric $(E,h)$. We define
\begin{equation}*_{4}=(\Omega\lrcorner)\circ*: \Omega^{0,2}(X,EndE)\rightarrow \Omega^{0,2}(X,EndE),
\nonumber \end{equation}
with $*_{4}^{2}=1$ and it splits the corresponding harmonic subspace into (anti-)self-dual parts.

The $DT_{4}$ equations are defined to be
\begin{equation}\label{DT4 equations} \left\{ \begin{array}{l}
  F^{0,2}_{+}=0 \\ F\wedge\omega^{3}=0 ,    % uses matrix to express the cpx ASD equations
\end{array}\right. \end{equation}
where the first equation is $F^{0,2}+*_{4}F^{0,2}=0$ and we assume $c_{1}(E)=0$ for simplicity in the moment map equation $ F\wedge\omega^{3}=0$.

We denote $\mathcal{M}^{DT_{4}}(X,g,[\omega],c,h)$ or simply $\mathcal{M}^{DT_{4}}_{c}$ to be the space of
gauge equivalence classes of solutions to the $DT_{4}$ equations on $E$ with $ch(E)=c$.

We take $\mathcal{M}_{c}^{bdl}$ to be the moduli space of slope-stable holomorphic bundles with fixed Chern character $c$.
By the Donaldson-Uhlenbeck-Yau's theorem \cite{UY}, we can identify it with the moduli space of gauge equivalence classes of
solutions to the holomorphic HYM equations
\begin{equation} \left\{ \begin{array}{l}
  F^{0,2}=0 \\ F\wedge\omega^{3}=0.    % uses matrix to express the cpx ASD equations
\end{array}\right.
\nonumber \end{equation}
By Lemma 4.1 \cite{caoleung}, if $ch_{2}(E)\in H^{2,2}(X,\mathbb{C})$,
then $F^{0,2}_{+}=0\Rightarrow F^{0,2}=0$. In particular, if $\mathcal{M}_{c}^{bdl}\neq\emptyset$,
then $\mathcal{M}_{c}^{DT_{4}}\cong\mathcal{M}_{c}^{bdl}$ as sets.

The comparison of analytic structures is given by
\begin{theorem}\label{mo mDT4}(Theorem 1.1 \cite{caoleung}) We assume $\mathcal{M}_{c}^{bdl}\neq\emptyset$ and
fix $d_{A}\in\mathcal{M}_{c}^{DT_{4}}$, then  \\
(1) there exists a Kuranishi map $\tilde{\tilde{\kappa}}$ of $\mathcal{M}_{c}^{bdl}$ at $\overline{\partial}_{A}$
(the (0,1) part of $d_{A}$) such that $\tilde{\tilde{\kappa}}_{+}$ is a Kuranishi map of $\mathcal{M}_{c}^{DT_{4}}$ at $d_{A}$, where
\begin{equation} \xymatrix@1{
\tilde{\tilde{\kappa}}_{+}=\pi_{+}(\tilde{\tilde{\kappa}}): H^{0,1}(X,EndE) \ar[r]^{\quad \quad \quad \tilde{\tilde{\kappa}}}
& H^{0,2}(X,EndE)\ar[r]^{\pi_{+}} & H^{0,2}_{+}(X,EndE) }  \nonumber \end{equation}
and $\pi_{+}$ is projection to the self-dual forms; \\
(2) the closed imbedding between analytic spaces possibly with non-reduced structures $\mathcal{M}_{c}^{bdl}\hookrightarrow \mathcal{M}_{c}^{DT_{4}}$
%\begin{equation}\mathcal{M}_{c}^{bdl}\hookrightarrow \mathcal{M}_{c}^{DT_{4}}  \nonumber \end{equation}
is also a homeomorphism between topological spaces.
\end{theorem}
\begin{remark}
By Proposition 10.10 \cite{caoleung}, the map $\tilde{\tilde{\kappa}}$ satisfies $Q_{Serre}(\tilde{\tilde{\kappa}},\tilde{\tilde{\kappa}})\geq0$,
where $Q_{Serre}$ is the Serre duality pairing on $H^{0,2}(X,EndE)$.
\end{remark}
To define Donaldson type invariants using $\mathcal{M}_{c}^{DT_{4}}$, we need to give it a good compactification (such that it carries
a deformation invariant fundamental class). For this purpose, we introduce $\mathcal{M}_{c}(X,\mathcal{O}_{X}(1))$ or
simply $\mathcal{M}_{c}$ to be the Gieseker moduli space of $\mathcal{O}_{X}(1)$-stable sheaves on $X$ with given Chern character $c$.
Motivated by Theorem \ref{mo mDT4}, we make the following definition.
\begin{definition}\label{gene DT4 moduli}(\cite{caoleung})
We call a $C^{\infty}$-scheme, $\overline{\mathcal{M}}^{DT_{4}}_{c}$ generalized $DT_{4}$ moduli space if there exists a homeomorphism
\begin{equation}\mathcal{M}_{c}\rightarrow\overline{\mathcal{M}}^{DT_{4}}_{c}  \nonumber \end{equation}
such that at each closed point of $\mathcal{M}_{c}$, say $\mathcal{F}$, $\overline{\mathcal{M}}^{DT_{4}}_{c}$ is locally isomorphic
to $\kappa_{+}^{-1}(0)$, where
\begin{equation}\kappa_{+}=\pi_{+}\circ\kappa:Ext^{1}(\mathcal{F},\mathcal{F})\rightarrow Ext^{2}_{+}(\mathcal{F},\mathcal{F}),
\nonumber \end{equation}
$\kappa$ is a Kuranishi map at $\mathcal{F}$
and $Ext^{2}_{+}(\mathcal{F},\mathcal{F})$ is a half dimensional real subspace of $Ext^{2}(\mathcal{F},\mathcal{F})$ on which the
Serre duality quadratic form is real and positive definite.
\end{definition}
\begin{remark}${}$ \\
1. The existence of generalized $DT_{4}$ moduli spaces was proved by Borisov-Joyce \cite{bj}. The authors proved their existence
as real analytic spaces in certain cases and defined the corresponding virtual fundamental classes \cite{cao},\cite{caoleung}. \\
2. For fixed data $(X,\mathcal{O}_{X}(1),c)$, generalized $DT_{4}$ moduli spaces may not be unique. However,
they all carry the same virtual fundamental class \cite{bj}.
\end{remark}
The proof of Borisov-Joyce's gluing result is divided into two parts. Firstly, they use good local models of $\mathcal{M}_{c}$, i.e.
local 'Darboux charts' in the sense of Brav, Bussi and Joyce \cite{bbj}. Then they chose the half dimensional real subspace
$Ext^{2}_{+}(\mathcal{F},\mathcal{F})$ appropriately and use partition of unity and homotopical algebra to glue $\kappa_{+}$.
We state the analytic version of BBJ's local Darboux charts as follows.
\begin{theorem}\label{B-side local Darboux thm}(Brav-Bussi-Joyce \cite{bbj} Corollary 5.20, see also Theorem 10.7 \cite{caoleung})  ${}$ \\
Let $\mathcal{M}_{c}$ be a Gieseker moduli space of stable sheaves on a compact Calabi-Yau 4-fold $X$.
Then for any closed point $\mathcal{F}\in\mathcal{M}_{c}$, there exists an analytic neighborhood $U_{\mathcal{F}}\subseteq\mathcal{M}_{c}$,
a holomorphic map near the origin
\begin{equation}\kappa: Ext^{1}(\mathcal{F},\mathcal{F})\rightarrow Ext^{2}(\mathcal{F},\mathcal{F})   \nonumber\end{equation}
such that $Q_{Serre}(\kappa,\kappa)=0$ and $\kappa^{-1}(0)\cong U_{\mathcal{F}}$ as complex analytic spaces possibly with non-reduced structures,
where $Q_{Serre}$ is the Serre duality pairing on $Ext^{2}(\mathcal{F},\mathcal{F})$.
\end{theorem}
\begin{proof}(See the Proof of Theorem 10.7 \cite{caoleung}) The point is that we can use Seidel-Thomas twists \cite{js},\cite{st}
transfer the problem to a problem on moduli spaces of holomorphic bundles and then notice $ch_{2}(E)\wedge \Omega=0$ for holomorphic bundle $E$,
where $\Omega$ is the holomorphic top form.
\end{proof}
To make sense of virtual fundamental classes of generalized $DT_{4}$ moduli spaces as homology classes in $\mathcal{M}_{c}$'s, one
would in general rely on Joyce's D-manifolds theory \cite{joyce1} or Fukaya-Oh-Ohta-Ono's theory of Kuranishi spaces \cite{fooo}
or Hofer's polyfolds theory \cite{hofer} (see \cite{yang} for a comparison between them).
Assuming this part, which is claimed by Borisov-Joyce, we have
\begin{theorem}(Borisov-Joyce \cite{bj}, \cite{joyce})\label{DT4 main thm} ${}$\\
Let $X$ be a complex projective Calabi-Yau 4-fold, and $\mathcal{M}_{c}$ be a Gieseker moduli space of stable sheaves which is compact
(it is true if the degree and rank of sheaves are coprime). Then there exists a generalized $DT_{4}$ moduli space $\overline{\mathcal{M}}^{DT_{4}}_{c}$. If we further assume $\overline{\mathcal{M}}^{DT_{4}}_{c}$ is orientable in the sense of D-manifold, then the virtual fundamental class of $\overline{\mathcal{M}}^{DT_{4}}_{c}$ exists and is a well-defined homology class, i.e.
\begin{equation}[\overline{\mathcal{M}}^{DT_{4}}_{c}]^{vir}\in H_{*}(\mathcal{M}_{c},\mathbb{Z}),   \nonumber\end{equation}
which coincides with earlier definitions of $DT_{4}$ virtual cycles (Definition 5.3, 5.12, 5.14 \cite{caoleung}).

We can furthermore define the $DT_{4}$ invariant by pairing the above cycle with $\mu$-map as in Definition 5.15 \cite{caoleung}. With appropriate choice of orientations, $DT_{4}$ invariants are invariant under deformations of complex structures of $X$.
\end{theorem}
Because of the Serre duality for $Ext^{*}(\mathcal{F},\mathcal{F})$, the existence of an orientation on a
generalized $DT_{4}$ moduli space $\overline{\mathcal{M}}^{DT_{4}}_{c}$ (in the sense of D-manifold) is equivalent to the existence of
a reduction of the structure group of $(\mathcal{L}_{X},Q_{Serre})$ to $SO(1,\mathbb{C})$, where $\mathcal{L}_{X}$ is the determinant line bundle with $\mathcal{L}_{X}|_{\mathcal{F}}\cong(\wedge^{top}Ext^{even}(\mathcal{F},\mathcal{F}))^{-1}\otimes\wedge^{top}Ext^{odd}
(\mathcal{F},\mathcal{F})$ and $Q_{Serre}$ is the Serre duality quadratic form on it.
\begin{theorem}(Theorem 2.2 \cite{caoleung3})\label{general existence of ori}
Let $X$ be a compact Calabi-Yau 4-fold with $H_{odd}(X,\mathbb{Z})=0$. For any Gieseker moduli space
$\mathcal{M}_{c}$ of stable sheaves, the structure group of $(\mathcal{L}_{X},Q_{Serre})$ can be reduced to $SO(1,\mathbb{C})$.
\end{theorem}
In the case when the Gieseker moduli space $\mathcal{M}_{c}$  of stable sheaves on $X$ is smooth (i.e. Kuranishi maps are zero), the obstruction
sheaf $Ob$ such that $Ob|_{\mathcal{F}}\cong Ext^{2}(\mathcal{F},\mathcal{F})$ is a bundle with Serre duality quadratic form.
There exists a real subbundle $Ob_{+}$ with positive definite quadratic form such that $Ob\cong Ob_{+}\otimes_{\mathbb{R}}\mathbb{C}$
as vector bundles with quadratic forms. We call $Ob_{+}$ the self-dual obstruction bundle. By Definition 5.12 \cite{caoleung},
the virtual fundamental class of $\overline{\mathcal{M}}_{c}^{DT_{4}}$ is the Poincar\'{e} dual of
the Euler class of the self-dual obstruction bundle (if it is orientable), i.e.
\begin{equation}\label{half euler class}[\overline{\mathcal{M}}^{DT_{4}}_{c}]^{vir}=PD(e(Ob_{+}))\in H_{*}(\mathcal{M}_{c},\mathbb{Z}). \end{equation}
This motivates later definitions of relative $DT_{4}$ invariants.

\subsection{Some basic facts in $DT_{3}$ theory}
%By viewing the $DT_{4}$ equations (\ref{DT4 equations}) as a gradient flow equation on $Y\times\mathbb{C}$, where $Y$ is a $CY_{3}$,
%we might be able to count gradient flow lines between holomorphic bundles on $Y$ and define the Morse homology.
Moduli spaces of simple sheaves on $CY_{3}$'s are locally critical points of holomorphic functions \cite{bbj}, \cite{js} and
we can consider the perverse sheaves of vanishing cycles of these functions. The expected cohomology which categorifies
$DT_{3}$ invariant is defined by first gluing these local perverse sheaves and then taking its hypercohomology.
\begin{theorem}(Brav-Bussi-Dupont-Joyce-Szendroi \cite{bbdjs}, Kiem-Li \cite{kl})\label{DT3 coho} ${}$ \\
Let $Y$ be a Calabi-Yau 3-fold over $\mathbb{C}$, and $\mathcal{M}$ a classical moduli scheme of simple coherent sheaves or simple complexes
of coherent sheaves on $Y$, with natural (symmetric) obstruction theory $\phi: \mathcal{E}^{\bullet}\rightarrow \mathbb{L}^{\bullet}_{\mathcal{M}}$.
Suppose we are given a square root of $det(\mathcal{E}^{\bullet})$, then there exists a perverse sheave $\mathcal{P}^{\bullet}_{\mathcal{M}}$
uniquely up to canonical isomorphism such that the Euler characteristic of its hypercohomology is the Donaldson-Thomas invariant \cite{th}.
\end{theorem}

\subsection{An overview of TQFT type structures in $DT_{4}$-$DT_{3}$ theories }
In this subsection, we give an overview of TQFT structures in gauge theories on Calabi-Yau 3-folds and 4-folds.

We take a smooth (Calabi-Yau) 3-fold $Y$ in a complex projective 4-fold $X^{+}$ as its anti-canonical divisor, and consider a moduli space
$\mathfrak{M}_{X^{+}}$ of stable bundles with fixed Chern classes on $X^{+}$ which has a well-defined restriction morphism
\begin{equation}r: \mathfrak{M}_{X^{+}}\rightarrow \mathfrak{M}_{Y}   \nonumber \end{equation}
to a moduli space of stable sheaves on $Y$. This would determine a class
$v(\mathfrak{M}_{X^{+}})\in \mathbb{H}^{*}(\mathcal{P}^{\bullet}_{\mathfrak{M}_{Y}})$.

Given another complex projective 4-fold $X^{-}$ which contains $Y$ as its anti-canonical divisor,
we form a singular space $X_{0}=X^{+}\cup_{Y}X^{-}$. When $X_{0}$ admits a smooth deformation $X_{t}$, $X_{t}$ will be a family of $CY_{4}$'s
provided that the normal bundle of $Y$ in $X^{\pm}$ is trivial. As the perverse sheaf in Theorem \ref{DT3 coho} is self-dual under the Verdier duality,
i.e. $\mathbb{D}_{\mathfrak{M}_{Y}}(\mathcal{P}^{\bullet}_{\mathfrak{M}_{Y}})\cong \mathcal{P}^{\bullet}_{\mathfrak{M}_{Y}}$ \cite{bbdjs}, we
can define $\langle v(\mathfrak{M}_{X^{+}}),v(\mathfrak{M}_{X^{-}})\rangle$ using the Verdier duality on
$\mathbb{H}^{*}(\mathfrak{M}_{Y},\mathcal{P}^{\bullet}_{\mathfrak{M}_{Y}})$ as long as $\mathfrak{M}_{Y}$ is compact.

Ignoring the stability issue and contributions from general coherent sheaves, we ask the following question which can be regarded as a \textit{complexification} of Chern-Simons-Donaldson-Floer TQFT structure for 3 and 4-manifolds (see Atiyah \cite{atiyah0}).
\begin{question}\label{DT4-DT3 gluing}
What is the relation between $DT_{4}$ invariants and relative $DT_{4}$ invariants, namely, comparing $DT_{4}(\mathfrak{M}_{X_{t}})$ with $\langle v(\mathfrak{M}_{X^{+}}),v(\mathfrak{M}_{X^{-}})\rangle$?
\end{question}

\section{Relative $DT_{4}$ invariants for holomorphic bundles}

\subsection{Definitions of relative $DT_{4}$ invariants}
In this section, we restrict to some good cases and define rigorously the relative $DT_{4}$ invariant mentioned before, i.e.
$v(\mathfrak{M}_{X^{+}})\in \mathbb{H}^{*}(\mathcal{P}^{\bullet}_{\mathfrak{M}_{Y}})$. To have a well-defined restriction map, in this section, we assume all Gieseker moduli spaces of semi-stable sheaves on 4-folds consist of slope-stable bundles only. This will serve as a
model for the later study of relative $DT_{4}$ invariants for ideal sheaves.

Let $Y$ be a smooth anti-canonical divisor of a smooth projective 4-fold $X$, $\mathfrak{M}_{X}$ be a Gieseker moduli space which is admissible with respect to $(X,Y)$ (see Definition \ref{admissible}), and
\begin{equation}r: \mathfrak{M}_{X}\rightarrow \mathfrak{M}_{Y} \nonumber \end{equation}
be the restriction morphism to a Gieseker moduli space on $Y$.
%\begin{assumption}\label{assumption of restriction map}
%Let $Y$ be a smooth anti-canonical divisor of a smooth complex projective 4-fold $X$. We take a Gieseker moduli space $\mathfrak{M}_{X}$ of semi-stable sheaves which consists of slope-stable bundles only. We assume it has a well-defined restriction morphism (see Theorem \ref{restriction thm} for its existence)
%\begin{equation}r: \mathfrak{M}_{X}\rightarrow \mathfrak{M}_{Y} \nonumber \end{equation}
%to a Gieseker moduli space of semi-stable sheaves on $Y$ which contains no strictly semi-stable sheaves.
%\end{assumption}
We recall the following criterion which ensures that we have such morphism $r$ in many cases.
\begin{theorem}(Flenner \cite{flenner})\label{restriction thm} ${}$ \\
Let $(X,\mathcal{O}_{X}(1))$ be a complex $n$-dimensional normal projective variety with $\mathcal{O}_{X}(1)$ very ample.
We take $\mathcal{F}$ to be a $\mathcal{O}_{X}(1)$-slope semi-stable torsion-free sheaf of rank $r$. $d$ and $1\leq c\leq n-1$ are integers such that
%\begin{equation}\frac{\left(\begin{array}{l}n+d \\ \quad d\end{array}\right)-cd-1}{d}>\deg(\mathcal{O}_{X}(1))\cdot\max(\frac{r^{2}-1}{4},1).
%\nonumber \end{equation}
\begin{equation}[\left(\begin{array}{l}n+d \\ \quad d\end{array}\right)-cd-1]/d>\deg(\mathcal{O}_{X}(1))\cdot\max(\frac{r^{2}-1}{4},1).
\nonumber \end{equation}
Then for a generic complete intersection $Y=H_{1}\cap\cdot\cdot\cdot\cap H_{c}$ with $H_{i}\in |\mathcal{O}_{X}(d)|$,
$\mathcal{F}|_{Y}\triangleq \mathcal{F}\otimes_{\mathcal{O}_{X}} \mathcal{O}_{Y}$ is $\mathcal{O}_{X}(1)|_{Y}$-slope semi-stable on $Y$.
\end{theorem}
\begin{remark}
For $X=\mathbb{P}^{4}$, $\mathcal{O}_{X}(1)=\mathcal{O}_{\mathbb{P}^{4}}(1)$ is very ample. We take $c=1$ and $d=5$, then
any rank $r\leq9$ semi-stable sheaf on $\mathbb{P}^{4}$ remains semi-stable when restricted to a generic quintic 3-fold inside.
\end{remark}
The deformation-obstruction theory associated to the restriction morphism $r$ is described by the following exact sequence.
\begin{lemma}\label{def-obs LES}
We take a stable bundle $E\in\mathfrak{M}_{X}$, and assume $Y$ is connected, then we have a long exact sequence,
\begin{equation}0\rightarrow H^{1}(X,EndE\otimes K_{X})\rightarrow H^{1}(X,EndE)\rightarrow H^{1}(Y,EndE|_{Y})
\rightarrow \nonumber \end{equation}
\begin{equation}\rightarrow H^{2}(X,EndE\otimes K_{X})\rightarrow H^{2}(X,EndE)\rightarrow H^{2}(Y,EndE|_{Y}) \rightarrow \nonumber \end{equation}
\begin{equation}\rightarrow H^{3}(X,EndE\otimes K_{X})\rightarrow H^{3}(X,EndE)\rightarrow 0. \quad\quad\quad \quad\quad\quad\quad \nonumber\end{equation}
\end{lemma}
\begin{proof}
We tensor $0\rightarrow\mathcal{O}_{X}(-Y)\rightarrow\mathcal{O}_{X}\rightarrow\mathcal{O}_{Y}\rightarrow0$ with $EndE$ and take its cohomology.
\end{proof}
We note that the transpose of the above sequence with respect to Serre duality pairing on $X$ and $Y$ remains the same (see also \cite{dt}). This will be the key property used in the following definitions of relative $DT_{4}$ invariants. \\
${}$ \\
\textbf{Case I: when $\mathfrak{M}_{Y}$ is of expected dim}. If we assume $H^{1}(Y,EndE|_{Y})=0$ for any $E\in \mathfrak{M}_{X}$,
then $H^{2}(Y,EndE|_{Y})=0$ and $\mathfrak{M}_{Y}$ contains components of finite points, labeled by $E_{1},...,E_{m}$ which come from
restrictions of bundles on $X$. We denote $\mathfrak{M}_{X,E_{i}}$ to be components of $\mathfrak{M}_{X}$ such that
$r(\mathfrak{M}_{X,E_{i}})=E_{i}$. By Lemma \ref{def-obs LES}, we have canonical isomorphisms
\begin{equation}H^{3}(X,EndE)^{*}\cong H^{1}(X,EndE), \quad H^{2}(X,EndE)^{*}\cong H^{2}(X,EndE). \nonumber \end{equation}
In fact, we apply Theorem 2.13 of \cite{calaque} and take the induced shifted symplectic structure on the stable loci as in \cite{bj},
$\mathfrak{M}_{X,E_{i}}$ has a $(-2)$-shifted symplectic structure in the sense of \cite{ptvv}. Analogs to Theorem
\ref{DT4 main thm}, we can define the relative $DT_{4}$ virtual cycle
$[\mathfrak{M}^{rel}_{X,E_{i}}]^{vir}\in H_{n}(\mathfrak{M}_{X,E_{i}},\mathbb{Z}_{2})$ with $n=1-\chi(X,EndE)$.
The cycle will be defined over integer if the associated D-manifold of $\mathfrak{M}_{X,E_{i}}$ is orientable.

The virtual dimension is not zero in general, we introduce a $\mu$-map as in Definition 5.15 \cite{caoleung}.
\begin{definition}\label{mu map relative}
We denote the universal sheaf of $\mathfrak{M}_{X}$ by $\mathfrak{F}\rightarrow\mathfrak{M}_{X}\times X$.

The $\mu$-map is
\begin{equation}\mu: H_{*}(X)\otimes \mathbb{Z}[x_{1},x_{2},...]\rightarrow H^{*}(\mathfrak{M}_{X}), \nonumber \end{equation}
\begin{equation}\mu(\gamma,P)=P(c_{1}(\mathfrak{F}),c_{2}(\mathfrak{F}),...)/\gamma. \nonumber \end{equation}
\end{definition}
Pairing virtual cycles with $\mu$-maps defines the relative invariants.
\begin{definition}\label{def of rel inv I}
Let $\mathfrak{M}_{X}$ be a Gieseker moduli space of semi-stable sheaves which is admissible with respect to $(X,Y)$ (see Definition \ref{admissible}), and $r: \mathfrak{M}_{X}\rightarrow \mathfrak{M}_{Y}$ be the restriction morphism.

We assume $H^{1}(Y,EndE|_{Y})=0$ for any $E\in \mathfrak{M}_{X}$, then the relative $DT_{4}$ invariant is a map
\begin{equation}\label{rel DT4 case 1}v(\mathfrak{M}_{X}): Sym^{*}\big(H_{*}(X,\mathbb{Z}) \otimes \mathbb{Z}[x_{1},x_{2},...]\big)
\rightarrow \mathbb{H}^{*}(\mathfrak{M}_{Y},\mathcal{P}^{\bullet}_{\mathfrak{M}_{Y}}) \end{equation}
such that
\begin{equation}v(\mathfrak{M}_{X})((\gamma_{1},P_{1}),(\gamma_{2},P_{2}),...) =\sum_{i=1}^{m}\langle[\mathfrak{M}^{rel}_{X,E_{i}}]^{vir},
\mu(\gamma_{1},P_{1})\cup \mu(\gamma_{2},P_{2})\cup... ,\rangle E_{i},\nonumber \end{equation}
where $\langle,\rangle$ denotes the natural pairing between homology and cohomology classes and $\{E_{i}\}_{1\leq i\leq m}$ are taken
as a basis of $H^{*}(r(\mathfrak{M}_{X}))$.
\end{definition}
${}$ \\
\textbf{Case II: when $\mathfrak{M}_{X}$ and $\mathfrak{M}_{Y}$ are smooth and $r$ is surjective}. We assume $\mathfrak{M}_{X}$ and
 $\mathfrak{M}_{Y}$ are smooth (i.e. all Kuranishi maps are zero) and the restriction map $r$ is surjective.
By Lemma \ref{def-obs LES}, we get a canonical isomorphism
\begin{equation}H^{2}(X,EndE)^{*}\cong H^{2}(X,EndE)\nonumber \end{equation}
which endows $H^{2}(X,EndE)$ a non-degenerate quadratic form, and a short exact sequence
\begin{equation}0\rightarrow H^{3}(X,EndE)^{*}\rightarrow H^{1}(X,EndE)\rightarrow H^{1}(Y,EndE|_{Y})\rightarrow 0.\nonumber \end{equation}
Counting dimensions, we have
\begin{equation}2h^{1}(X,EndE)-h^{2}(X,EndE)=h^{1}(Y,EndE|_{Y})-\chi(X,EndE)+1, \quad \nonumber \end{equation}
which is a constant on components of $\mathfrak{M}_{X}$ by assumptions. Similar to (\ref{half euler class}),
the self-dual subbundle of the obstruction bundle $Ob_{\mathfrak{M}_{X}}$ exists.
If it is also orientable, we define the relative $DT_{4}$ virtual cycle
$[\mathfrak{M}_{X}^{rel}]^{vir}\in H_{n}(\mathfrak{M}_{X},\mathbb{Z})$ to be the Euler class of the self-dual obstruction bundle,
where $n=h^{1}(Y,EndE|_{Y})-\chi(X,EndE)+1$.
\begin{definition}\label{def of rel inv II}
Let $\mathfrak{M}_{X}$ be a Gieseker moduli space of semi-stable sheaves which is admissible with respect to $(X,Y)$ (see Definition \ref{admissible}), and $r: \mathfrak{M}_{X}\rightarrow \mathfrak{M}_{Y}$ be the restriction morphism.

We assume $r$ is surjective between smooth moduli spaces, then the relative $DT_{4}$ invariant is a map
\begin{equation}\label{rel DT4 case 2}v(\mathfrak{M}_{X}): Sym^{*}\big(H_{*}(X,\mathbb{Z}) \otimes \mathbb{Z}[x_{1},x_{2},...]\big)
\rightarrow  H^{*}(\mathfrak{M}_{Y}) \end{equation}
such that
\begin{equation}v(\mathfrak{M}_{X})((\gamma_{1},P_{1}),(\gamma_{2},P_{2}),...)(\alpha)=\langle[\mathfrak{M}_{X}^{rel}]^{vir},
(r^{*}\alpha)\cup\mu(\gamma_{1},P_{1})\cup \mu(\gamma_{2},P_{2})\cup... ,\rangle,\nonumber \end{equation}
where $\alpha\in H^{*}(\mathfrak{M}_{Y})$ and we identify $H^{*}(\mathfrak{M}_{Y})\cong H^{*}(\mathfrak{M}_{Y})^{*}$ via Poincar\'{e} pairing,
$\langle,\rangle$ denotes the natural pairing between homology and cohomology classes.
\end{definition}
${}$ \\
\textbf{Case III: when $\mathfrak{M}_{X}$ and $\mathfrak{M}_{Y}$ are smooth and $r$ is injective}.
We assume $\mathfrak{M}_{X}$ and $\mathfrak{M}_{Y}$ are smooth (i.e. all Kuranishi maps are zero) and the restriction map $r$ is injective. By Lemma \ref{def-obs LES}, we get $H^{3}(X,EndE)=0$ and an exact sequence
\begin{equation}0\rightarrow H^{1}(X,EndE)\rightarrow H^{1}(Y,EndE|_{Y})\rightarrow H^{2}(X,EndE)^{*}\rightarrow\nonumber \end{equation}
\begin{equation}\rightarrow  H^{2}(X,EndE)\rightarrow H^{1}(Y,EndE|_{Y})^{*} \rightarrow H^{1}(X,EndE)^{*}\rightarrow 0.\nonumber \end{equation}
This determine a surjective map
\begin{equation}s: Ob_{\mathfrak{M}_{X}}\twoheadrightarrow \mathcal{N}^{*}_{\mathfrak{M}_{X}/\mathfrak{M}_{Y}} \nonumber \end{equation}
and a non-degenerate quadratic form on the reduced bundle $Ob_{\mathfrak{M}_{X}}^{red}\triangleq Ker(s)$,
where $Ob_{\mathfrak{M}_{X}}$ is the obstruction bundle of $\mathfrak{M}_{X}$ with $Ob_{\mathfrak{M}_{X}}|_{E}=H^{2}(X,EndE)$
and $\mathcal{N}^{*}_{\mathfrak{M}_{X}/\mathfrak{M}_{Y}}$ is the conormal bundle of $\mathfrak{M}_{X}$ inside $\mathfrak{M}_{Y}$.

Then if the self-dual subbundle of $Ob^{red}_{\mathfrak{M}_{X}}$ is orientable,
we define the relative $DT_{4}$ virtual cycle
$[\mathfrak{M}_{X}^{rel}]^{vir}\in H_{n}(\mathfrak{M}_{X},\mathbb{Z})$ to be the Euler class of it, where the virtual dimension is
$n=2h^{1}(X,EndE)-\big(h^{2}(X,EndE)-codim_{\mathfrak{M}_{Y}}(\mathfrak{M}_{X})\big)=h^{1}(Y,EndE|_{Y})-\chi(X,EndE)+1$. Note that
when $\mathfrak{M}_{X}$ is smooth, $r$ is injective and a neighbourhood of $r(\mathfrak{M}_{X})\subseteq \mathfrak{M}_{Y}$ is smooth, we could also define $[\mathfrak{M}_{X}^{rel}]^{vir}$ in a similar way.
\begin{definition}\label{def of rel inv IV}
Let $\mathfrak{M}_{X}$ be a Gieseker moduli space of semi-stable sheaves which is admissible with respect to $(X,Y)$ (see Definition \ref{admissible}), and $r: \mathfrak{M}_{X}\rightarrow \mathfrak{M}_{Y}$ be the restriction morphism.

We assume $r$ is injective between smooth moduli spaces (at lease when restricted to a neighbourhood of $r(\mathfrak{M}_{X})$ in $\mathfrak{M}_{Y}$),
then the relative $DT_{4}$ invariant is a map
\begin{equation}\label{rel DT4 case 4}v(\mathfrak{M}_{X}): Sym^{*}\big(H_{*}(X,\mathbb{Z}) \otimes \mathbb{Z}[x_{1},x_{2},...]\big)
\rightarrow H^{*}(\mathfrak{M}_{Y}) \end{equation}
such that
\begin{equation}v(\mathfrak{M}_{X})((\gamma_{1},P_{1}),(\gamma_{2},P_{2}),...)(\alpha)=\langle[\mathfrak{M}_{X}^{rel}]^{vir},
(r^{*}\alpha)\cup\mu(\gamma_{1},P_{1})\cup \mu(\gamma_{2},P_{2})\cup... ,\rangle,\nonumber \end{equation}
where $\alpha\in H^{*}(\mathfrak{M}_{Y})$ and we identify $H^{*}(\mathfrak{M}_{Y})\cong H^{*}(\mathfrak{M}_{Y})^{*}$ via Poincar\'{e} pairing,
 $\langle,\rangle$ denotes the natural pairing between homology and cohomology classes.
\end{definition}
\begin{remark}
It is easy to check Definition \ref{def of rel inv I}, \ref{def of rel inv II} and \ref{def of rel inv IV}
are all compatible.
\end{remark}
In general, one could consider moduli spaces of complexes of simple sheaves \cite{lieb} on a complex projective 4-fold $X$ and
resolve complexes of sheaves by complexes of holomorphic bundles, then there will be a natural restriction morphism to a moduli of
simple bundles on an anti-canonical divisor of $X$. A similar long exact sequence in Lemma \ref{def-obs LES}
still works and we could study virtual cycle constructions as in \textbf{Cases I-III}.

${}$ \\
\textbf{Endomorphisms of $DT_{3}$ cohomologies from relative $DT_{4}$ invariants}.
By considering the trace-free version of Lemma \ref{def-obs LES}, the above definitions extend to any disconnected divisor $Y\subseteq X$.
We are particularly interested in the case when $X=Y_{1}\times \mathbb{P}^{1}$, where $Y_{1}$ is a compact Calabi-Yau 3-fold.
Then $Y=(Y_{1}\times0)\sqcup (Y_{1}\times\infty)$ will be a smooth anti-canonical divisor of $X$ and the relative $DT_{4}$ invariant
in general is a map
\begin{equation}v(\mathfrak{M}_{X}): Sym^{*}\big(H_{*}(X,\mathbb{Z}) \otimes \mathbb{Z}[x_{1},x_{2},...]\big)
\rightarrow \mathbb{H}^{*}(\mathfrak{M}_{Y},\mathcal{P}^{\bullet}_{\mathfrak{M}_{Y}})\cong
\mathbb{H}^{*}(\mathfrak{M}_{Y_{1}},\mathcal{P}^{\bullet}_{\mathfrak{M}_{Y_{1}}})\otimes
\mathbb{H}^{*}(\mathfrak{M}_{Y_{1}},\mathcal{P}^{\bullet}_{\mathfrak{M}_{Y_{1}}}). \nonumber \end{equation}
By the Verdier duality and
$\mathbb{D}_{\mathfrak{M}_{Y_{1}}}(\mathcal{P}^{\bullet}_{\mathfrak{M}_{Y_{1}}})\cong \mathcal{P}^{\bullet}_{\mathfrak{M}_{Y_{1}}}$, we have
\begin{equation}\mathbb{H}^{*}(\mathfrak{M}_{Y_{1}},\mathcal{P}^{\bullet}_{\mathfrak{M}_{Y_{1}}})\cong
\mathbb{H}^{*}(\mathfrak{M}_{Y_{1}},\mathcal{P}^{\bullet}_{\mathfrak{M}_{Y_{1}}})^{*}.  \nonumber \end{equation}
Thus a relative $DT_{4}$ invariant of $(Y_{1}\times \mathbb{P}^{1},Y_{1}\times\{0,\infty\})$
determines some endomorphisms of the $DT_{3}$ cohomology
\begin{equation}v(\mathfrak{M}_{X}): Sym^{*}\big(H_{*}(Y_{1})[t]/(t^{2}) \otimes \mathbb{Z}[x_{1},x_{2},...]\big)
\rightarrow End_{\mathbb{C}}\big(\mathbb{H}^{*}(\mathfrak{M}_{Y_{1}},\mathcal{P}^{\bullet}_{\mathfrak{M}_{Y_{1}}})\big)   \nonumber \end{equation}
for any Calabi-Yau 3-fold $Y_{1}$. In general, the above endomorphisms should be used to establish
the gluing formula mentioned in Question \ref{DT4-DT3 gluing} (see \cite{km} for the real 4-3 dimensional picture).

\subsection{Li-Qin's examples}
We have Li-Qin's examples when conditions in \textbf{Case II, III} are satisfied \cite{lq}.
Let $Y$ be a generic smooth hyperplane section in $X=\mathbb{P}^{1}\times\mathbb{P}^{3}$ of bi-degree $(2,4)$,
\begin{equation}c=[1+(-1,1)]\cdot[1+(\epsilon_{1}+1,\epsilon_{2}-1)],\nonumber \end{equation}
\begin{equation}c|_{Y}=[1+(-1,1)|_{Y}]\cdot[1+(\epsilon_{1}+1,\epsilon_{2}-1)|_{Y}],\nonumber \end{equation}
\begin{equation}k=(1+\epsilon_{1})\left(\begin{array}{l}5-\epsilon_{2} \\ \quad 3\end{array}\right), \quad \epsilon_{1},\epsilon_{2}=0,1,
\quad L_{r}=\mathcal{O}_{\mathbb{P}^{1}\times\mathbb{P}^{3}}(1,r).  \nonumber\end{equation}
We denote $\mathfrak{M}_{c}(L_{r})$ to be the moduli space of $L_{r}$-slope stable rank-2 bundles on $X$
with a Chern class $c$ and $\overline{\mathfrak{M}}_{c|_{Y}}(L_{r}|_{Y})$ to be the moduli space of Gieseker $L_{r}|_{Y}$-semistable
rank-2 torsion-free sheaves
on $Y$ with Chern class $c|_{Y}$.
\begin{lemma}(Li-Qin \cite{lq}) \label{lq example} ${}$ \\
(i) If
\begin{equation}\frac{4(2-\epsilon_{2})}{2+2\epsilon_{1}+\epsilon_{2}}<r<\frac{4(2-\epsilon_{2})}{\epsilon_{1}\epsilon_{2}},
\nonumber\end{equation}
then $\overline{\mathfrak{M}}_{c|_{Y}}(L_{r}|_{Y})\cong\mathbb{P}^{k}$ and consists of rank-2 stable bundles.
Furthermore, the restriction map
\begin{equation}r: \mathfrak{M}_{c}(L_{r})\rightarrow \overline{\mathfrak{M}}_{c|_{Y}}(L_{r}|_{Y}) \nonumber\end{equation}
is well-defined and an isomorphic between projective varieties. \\
${}$ \\
(ii) If $0<r<\frac{4(2-\epsilon_{2})}{2+2\epsilon_{1}+\epsilon_{2}}$,
then $\mathfrak{M}_{c}(L_{r})$ and $\overline{\mathfrak{M}}_{c|_{Y}}(L_{r}|_{Y})$ are empty.
\end{lemma}
\begin{proposition}\label{lq example compute}
In the above example, for any stable bundle $E\in\mathfrak{M}_{c}(L_{r})$ on $X=\mathbb{P}^{1}\times\mathbb{P}^{3}$,
\begin{equation}Ext^{1}_{X}(E,E)\cong Ext^{1}_{Y}(E|_{Y},E|_{Y})\cong Ext^{2}_{Y}(E|_{Y},E|_{Y})^{*}\cong\mathbb{C}^{k}, \nonumber\end{equation}
\begin{equation}Ext^{i}_{X}(E,E)=0, \quad \textrm{if} \textrm{ } i\geq2. \nonumber\end{equation}
The relative $DT_{4}$ virtual cycle $[\mathfrak{M}^{rel}_{c}(L_{r})]^{vir}=[\mathfrak{M}_{c}(L_{r})]\in H_{2k}(\mathfrak{M}_{c}(L_{r}),\mathbb{Z})$.
\end{proposition}
%and the relative invariant satisfies
%\begin{equation}v(\mathfrak{M}_{c}(L_{r}))((\gamma_{1},0),(\gamma_{2},0),...)=(\int_{\mathbb{P}^{k}}H^{k})\cdot H^{k}=H^{k}\in
%H^{2k}(\mathbb{P}^{k},\mathbb{Z}), \nonumber \end{equation}
%where $H\in H^{2}(\mathbb{P}^{k},\mathbb{Z})$ is the hyperplane class and we identify
%$\overline{\mathfrak{M}}_{c|_{Y}}(L_{r}|_{Y})\cong\mathbb{P}^{k}$.

\section{Computational examples of relative $DT_{4}$ invariants for ideal sheaves}
In the above section, we studied relative $DT_{4}$ invariants for holomorphic bundles. To formulate the gluing formula,
we need to have a good understanding of how stable sheaves could be degenerated into union of stable sheaves on
irreducible components of degenerated varieties. At this moment, we will restrict ourselves to the ideal sheaves case where degenerations have simpler behavior.

We take a simple degeneration $\pi: \mathcal{X}\rightarrow C$ of projective manifolds over a pointed smooth curve $(C,0\in C)$, i.e. (1)
$\mathcal{X}$ is smooth, $\pi$ is projective and smooth away from the central fiber $X_{0}=\pi^{-1}(0)$, (2) $X_{0}$ is a union of two
smooth irreducible components $X_{+}$, $X_{-}$ intersecting transversally along a smooth divisor $Y$. When generic fibers $X_{t}$'s are
Calabi-Yau 4-folds and $Y$ is an anti-canonical divisor of $X_{+}$, $X_{-}$, we will study relative
$DT_{4}$ invariants of ideal sheaves for pairs $(X_{\pm},Y)$ and discuss their relations with $DT_{4}$ invariants of $X_{t}$, $t\neq0$.
The basic technique is the degeneration method developed by J. Li and B. Wu \cite{li1}, \cite{li2}, \cite{liwu}, \cite{wu}. \\

Li-Wu's construction will be recalled in the appendix and the associated obstruction theory is studied there. In this section, we concentrate on computational examples of relative $DT_{4}$ virtual cycles for ideal sheaves based on the extension (Definition \ref{rel DT4 vc for ideal sheaves}) of constructions for bundles (\textbf{Cases I-III}).
\begin{example}(Ideal sheaves of one point)\label{ideal sheaves of one pt}

We take a compact simply connected 4-fold\footnote{Note that any compact Fano manifold is simply connected.} $X_{+}$ which contains a smooth Calabi-Yau 3-fold $Y$ as its anti-canonical divisor.
We consider the moduli space $I_{1}(X_{+},0)$ of structure sheaves of one point (it is equivalent to consider ideal sheaves of one point)
which has a well-defined restriction map to $Y$ if points sit inside $X_{+}\backslash Y$.

To extend the map to the whole moduli space, we introduce Li-Wu's expanded pair, i.e. we consider $X_{+}[1]_{0}=X_{+}\cup_{Y}\Delta$, $Y[1]_{0}\subseteq\Delta$, where $\Delta\cong \mathbb{P}(\mathcal{O}_{Y}\oplus \mathcal{N}_{Y/X_{+}})$ and form the moduli space $I_{1}(X_{+}[1]_{0},0)$ ($\cong X_{+}$) of relative structure sheaf of one point, which is the union of $X_{+}\backslash Y$ with the $\mathbb{C}^{*}$-equivalence classes of points
in $\Delta\backslash((Y\times0)\cup (Y\times\infty))$. By the Koszul resolution and Serre duality, we have canonical isomorphism
\begin{equation}Ext^{*}(\mathcal{O}_{P},\mathcal{O}_{P})\cong \wedge^{*}(TX_{+}|_{P}).  \nonumber \end{equation}
Then the obstruction bundle $Ob=\wedge^{2}TX_{+}$ has a non-degenerate quadratic form only when it is restricted to $X_{+}\backslash Y$
and $\Delta\backslash((Y\times0)\cup (Y\times\infty))$, but they do not glue to become a
non-degenerate quadratic form on $Ob\rightarrow X_{+}$ as $c_{1}(Ob)\neq0$.

If we assume $K_{X_{+}}$ has a square root $K_{X_{+}}^{\frac{1}{2}}$ (see the appendix), and form $\widetilde{Ob}\triangleq\wedge^{2}TX_{+}\otimes K_{X_{+}}^{\frac{1}{2}}$, then there exists a non-degenerate quadratic form
\begin{equation}(\wedge^{2}TX_{+}\otimes K_{X_{+}}^{\frac{1}{2}})\otimes (\wedge^{2}TX_{+}\otimes K_{X_{+}}^{\frac{1}{2}})\rightarrow
\mathcal{O}_{X_{+}}. \nonumber \end{equation}
As $\pi_{1}(X)=0$, the self-dual obstruction bundle $\widetilde{Ob}_{+}$ is orientable, and Conjectures \ref{conj0}, \ref{conj1} hold for this case. To calculate the Euler class $e(\widetilde{Ob}_{+})$, we consider the case when $X_{+}=Y_{1}\times \mathbb{P}^{1}$ and $Y=(Y_{1}\times0)\sqcup (Y_{1}\times\infty)$, where $Y_{1}$ is a smooth Calabi-Yau 3-fold. Then
\begin{equation}
\wedge^{2}TX_{+}\otimes K_{X_{+}}^{\frac{1}{2}}\cong(\wedge^{2}TY_{1}\boxtimes\mathcal{O}_{\mathbb{P}^{1}}(-1))\oplus
(TY_{1}\boxtimes\mathcal{O}_{\mathbb{P}^{1}}(1)), \nonumber \end{equation}
where both factors are maximal isotropic subbundles of $\widetilde{Ob}_{+}$. By \cite{eg} or Lemma 5.13 \cite{caoleung},
\begin{equation} e(\widetilde{Ob}_{+})=\pm\big(c_{3}(Y_{1})+c_{2}(Y_{1})\cdot c_{1}(\mathcal{O}_{\mathbb{P}^{1}}(1))\big)=
\pm\big(c_{3}(X_{+})-\frac{1}{2}c_{1}(X_{+})\cdot c_{2}(X_{+})\big), \nonumber \end{equation}
where the sign depends on the orientation of $\widetilde{Ob}_{+}$. The relative $DT_{4}$
virtual cycle for structure sheaves of one point is the Poincar\'{e} dual of $e(\widetilde{Ob}_{+})$ (\ref{half euler class}).
\end{example}
We then consider examples of ideal sheaves of curves.
\begin{example}\label{eg 1}
Let $Q\subseteq \mathbb{P}^{4}$ be a smooth generic quintic 3-fold. We take $X_{+}=Q\times \mathbb{P}^{1}$ with an anti-canonical divisor $Y=(Q\times 0)\sqcup (Q\times \infty)$. Then $H_{2}(X_{+},\mathbb{Z})\cong H_{2}(Q,\mathbb{Z})\oplus H_{2}(\mathbb{P}^{1},\mathbb{Z})
\cong \mathbb{Z}\oplus\mathbb{Z}$.

(i) We first fix the curve class to be the generator $[H]\in H_{2}(\mathbb{P}^{1},\mathbb{Z})\subseteq H_{2}(X_{+},\mathbb{Z})$.
The moduli  space $I_{0}(X_{+},[H])$ ($\cong Q$) of ideal sheaves of curves with Chern character
$c=(1,0,0,-PD([H]),0)$ consists of ideal sheaves of curves of type $\{pt\}\times \mathbb{P}^{1}$, $pt\in Q$.
We have a restriction morphism
\begin{equation}r: I_{0}(X_{+},[H])\rightarrow I_{1}(Q,0)\times I_{1}(Q,0)\cong Q\times Q, \nonumber \end{equation}
\begin{equation}I_{\{pt\}\times \mathbb{P}^{1}}\mapsto (I_{pt},I_{pt}), \nonumber \end{equation}
to the moduli space of ideal sheaves of one point in $Y$. $r$ is the diagonal embedding if we identify $I_{0}(X_{+},[H])\cong Q$.
By direct calculations, for any $I_{C}\in I_{0}(X_{+},[H])$, we have
\begin{equation}Ext^{i}_{X_{+}}(I_{C},I_{C})\cong \mathbb{C}^{3},\textrm{ } i=1,2, \quad Ext^{3}_{X_{+}}(I_{C},I_{C})\cong\mathbb{C},
\quad Ext^{4}_{X_{+}}(I_{C},I_{C})=0. \nonumber \end{equation}
Analogs to Lemma \ref{fund exact seq for dege }, we have a long exact sequence,
\begin{equation}0\rightarrow Ext^{1}_{X_{+}}(I_{C},I_{C})\rightarrow Ext^{1}_{Y}(I_{C}|_{Y},I_{C}|_{Y})
\rightarrow Ext^{2}_{X_{+}}(I_{C},I_{C}\otimes K_{X_{+}})\nonumber \end{equation}
\begin{equation}\rightarrow  Ext^{2}_{X_{+}}(I_{C},I_{C})\rightarrow Ext^{2}_{Y}(I_{C}|_{Y},I_{C}|_{Y}) \rightarrow
Ext^{3}_{X_{+}}(I_{C},I_{C}\otimes K_{X_{+}})\rightarrow 0. \nonumber \end{equation}
This determines a surjective morphism
\begin{equation}s: Ob\twoheadrightarrow \mathcal{N}_{Q/Q\times Q}, \nonumber \end{equation}
from the obstruction bundle $Ob$ with $Ob|_{I_{C}}=Ext^{2}_{X_{+}}(I_{C},I_{C})$
to the conormal bundle of $I_{0}(X_{+},[H])$ in $I_{1}(Q,0)$.
Furthermore, $rk(Ob)=codim(Q,Q\times Q)=3$ and conditions in Definition \ref{rel DT4 vc for ideal sheaves} are satisfied. The relative $DT_{4}$ virtual cycle is the usual fundamental class of the moduli space,
i.e. $[I_{0}(X_{+},[H])]\in H_{6}(I_{0}(X_{+},[H]),\mathbb{Z})$.  \\

(ii) If we fix the curve class to be the generator $[H_{Q}]\in H_{2}(Q,\mathbb{Z})\subseteq H_{2}(X_{+},\mathbb{Z})$,
the moduli space $I_{1}(X_{+},[H_{Q}])$ of ideal sheaves of curves in $X_{+}=Q\times \mathbb{P}^{1}$ with Chern character $c=(1,0,0,-PD([H_{Q}]),-1)$
can be identified with the product of $\mathbb{P}^{1}$ with the moduli space of primitive rational curves
in $Q$ (which consists of $2875$ rigid curves for a generic $Q\subseteq \mathbb{P}^{4}$ \cite{coxkatz}), i.e.
\begin{equation}I_{1}(X_{+},[H_{Q}])\cong \bigsqcup_{2875}\mathbb{P}^{1}. \nonumber \end{equation}

Curves in $\bigsqcup_{2875}\mathbb{C}^{*}$ have well-defined restriction to trivial line bundles on $(Q\times 0)\sqcup
(Q\times \infty)$. For curves in $\bigsqcup_{2875}\{0,\infty \}\subseteq I_{1}(X_{+},[H_{Q}])$, we introduce Li-Wu's expanded pair
to define the restriction map. We denote $X_{+}[1]_{0}=\Delta_{-1}\cup X_{+}\cup\Delta_{1}$, where $\Delta_{\pm1}\cong Q\times\mathbb{P}^{1}$,
and consider the moduli space $I_{1}(X_{+}[1]_{0},[H_{Q}])$ of relative ideal sheaves of curves (normal to the divisor $(Q\times 0)\sqcup
(Q\times \infty)$). $I_{1}(X_{+}[1]_{0},[H_{Q}])$ is the union of $\bigsqcup_{2875}\mathbb{C}^{*}$ with $\mathbb{C}^{*}$-equivalence
classes of curves inside $\Delta_{\pm1}\backslash((Q\times 0)\sqcup (Q\times \infty))$, i.e.
\begin{equation}I_{1}(X_{+}[1]_{0},[H_{Q}])\cong \bigsqcup_{2875}\mathbb{P}^{1}.  \nonumber \end{equation}
We then have a restriction map
\begin{equation}I_{1}(X_{+}[1]_{0},[H_{Q}])\rightarrow \{\mathcal{O}_{Q\times 0}\}\sqcup \{\mathcal{O}_{Q\times \infty}\}, \nonumber \end{equation}
\begin{equation}I_{C}\mapsto I_{C}|_{(Q\times 0)\sqcup (Q\times \infty)}=(\mathcal{O}_{Q\times 0},\mathcal{O}_{Q\times \infty})
\nonumber \end{equation}
to the moduli space of trivial line bundles on $(Q\times 0)\sqcup (Q\times \infty)\subseteq \Delta_{-1}\sqcup \Delta_{1}$.
By direct calculations, for any $I_{C}\in I_{1}(X_{+}[1]_{0},[H_{Q}])$, we have
\begin{equation}Ext^{1}_{X_{+}[1]_{0}}(I_{C},I_{C})\cong\mathbb{C}, \quad Ext^{3}_{X_{+}[1]_{0}}(I_{C},I_{C})\cong\mathbb{C}^{2},
\quad Ext^{i}_{X_{+}[1]_{0}}(I_{C},I_{C})=0, \textrm{  } i=2,4. \nonumber \end{equation}
Conditions in Definition \ref{rel DT4 vc for ideal sheaves} are satisfied and the relative $DT_{4}$ virtual cycle is the usual fundamental class of the moduli space, i.e.
$[I_{1}(X_{+}[1]_{0},[H_{Q}])]\in H_{2}(I_{1}(X_{+}[1]_{0},[H_{Q}]),\mathbb{Z})$.
\end{example}
We give a gluing formula of relative $DT_{4}$ invariants for the above example.
\begin{example}\label{ex gluing}
In Example \ref{eg 1} (ii), we consider $X_{0}=X_{+}\cup_{Y}X_{-}$ and its smoothing $X_{t}=Q\times \mathbb{T}^{2}$, where $X_{\pm}\cong Q\times \mathbb{P}^{1}$ and $Y=(Q\times 0)\sqcup (Q\times \infty)$. The moduli space $I_{1}(X_{\pm}[1]_{0},[H_{Q}])$ of relative ideal sheaves of curves satisfies
\begin{equation}I_{1}(X_{\pm}[1]_{0},[H_{Q}])\cong \bigsqcup_{2875}\mathbb{P}^{1}.  \nonumber \end{equation}
The relative $DT_{4}$ virtual cycle is usual fundamental class of $I_{1}(X_{\pm}[1]_{0},[H_{Q}])$. Meanwhile, the moduli space $I_{1}(X_{t},[H_{Q}])$ of ideal sheaves of curves in $X_{t}=Q\times \mathbb{T}^{2}$ with curve class $[H_{Q}]$ satisfies
\begin{equation}I_{1}(X_{t},[H_{Q}])\cong \bigsqcup_{2875}\mathbb{T}^{2}, \nonumber \end{equation}
and its $DT_{4}$ virtual cycle is the usual fundamental class of $I_{1}(X_{t},[H_{Q}])$ \cite{caoleung}. Under the homologous relation $X_{0}\sim X_{t}$, we have $\mathbb{P}^{1}\cup_{\{0,\infty\}}\mathbb{P}^{1}\sim \mathbb{T}^{2}$, then
the gluing formula is expressed by pairing $\mu$-map with these cycles (see Theorem \ref{dege formula} for such a formula).

\end{example}
\begin{example}(Generic quintic in $\mathbb{P}^{4}$)\label{ex quintic}

We take $X_{+}=\mathbb{P}^{4}$ which contains a smooth generic quintic 3-fold $Y=Q$ as its anti-canonical divisor, and then
$H_{2}(X_{+},\mathbb{Z})\cong H_{2}(Q,\mathbb{Z})\cong\mathbb{Z}$. We consider the primitive curve class $[H]\in H_{2}(X_{+},\mathbb{Z})$ and
ideal sheaves of curves representing this class have Chern character $c=(1,0,0,-PD([H]),\frac{3}{2})$. We denote their moduli space by
$I_{\frac{3}{2}}(X_{+},[H])$ ($\cong Gr(2,5)$).
The generic quintic $Q$ contains $2875$ rigid degree $1$ rational curves and $I_{\frac{3}{2}}(X,[H])$ contains a finite subset $S$ with $2875$ points. $I_{\frac{3}{2}}(X,[H])\backslash S$ has a well-defined restriction morphism to $Hilb^{5}(Q)$.
To extend the morphism across those $2875$ points, we introduce Li-Wu's expanded pair.

We denote $X_{+}[1]_{0}=X_{+}\cup\Delta_{1}$, $Y[1]_{0}(\cong Q)\subseteq\Delta_{1}$, where $\Delta_{1}\cong \mathbb{P}(\mathcal{O}_{Q}\oplus
\mathcal{O}_{\mathbb{P}^{4}}(5)|_{Q})$, and consider the moduli space $I_{\frac{3}{2}}(X_{+}[1]_{0},[H])$ of
relative ideal sheaves of curves (normal to the divisor $Q$). Geometrically, it
is the blow up of $I_{\frac{3}{2}}(X_{+},[H])$ along those $2875$ points, i.e.
\begin{equation}I_{\frac{3}{2}}(X_{+}[1]_{0},[H])\cong Bl_{S}(Gr(2,5)),  \nonumber \end{equation}
and each exceptional divisor corresponds to a $Hilb^{5}(\mathbb{P}^{1})$ ($\cong\mathbb{P}^{5}$) for each $\mathbb{P}^{1}\subseteq Q$.
We then have
a restriction morphism
\begin{equation}I_{\frac{3}{2}}(X_{+}[1]_{0},[H])\rightarrow Hilb^{5}(Y[1]_{0}),  \nonumber \end{equation}
\begin{equation}I_{C}\mapsto I_{C}|_{Y[1]_{0}},  \nonumber \end{equation}
which is injective with smooth image. By direct calculations, we have
\begin{equation}Ext^{1}_{X_{+}[1]_{0}}(I_{C},I_{C})\cong \mathbb{C}^{6}, \quad Ext^{2}_{X_{+}[1]_{0}}(I_{C},I_{C})\cong \mathbb{C}^{9}, \quad
Ext^{3}_{X_{+}[1]_{0}}(I_{C},I_{C})=0, \nonumber \end{equation}
and a long exact sequence
\begin{equation}0\rightarrow Ext^{1}_{X_{+}[1]_{0}}(I_{C},I_{C})\rightarrow Ext^{1}_{Y[1]_{0}}(I_{C}|_{Y},I_{C}|_{Y})
\rightarrow Ext^{2}_{X_{+}[1]_{0}}(I_{C},I_{C})^{*}\rightarrow \nonumber \end{equation}
\begin{equation}\rightarrow  Ext^{2}_{X_{+}[1]_{0}}(I_{C},I_{C})\rightarrow Ext^{1}_{Y[1]_{0}}(I_{C}|_{Y},I_{C}|_{Y})^{*}
\rightarrow Ext^{1}_{X_{+}[1]_{0}}(I_{C},I_{C})^{*}\rightarrow 0. \nonumber \end{equation}
As $I_{C}|_{Y}\in Hilb^{5}(Y[1]_{0})$ is a smooth point with $Ext^{1}_{Y[1]_{0}}(I_{C}|_{Y},I_{C}|_{Y})\cong \mathbb{C}^{15}$, conditions in Definition \ref{rel DT4 vc for ideal sheaves} are satisfied. The relative $DT_{4}$ virtual cycle is
the usual fundamental class of the moduli space $I_{\frac{3}{2}}(X_{+}[1]_{0},[H])\cong Bl_{S}(Gr(2,5))$.
\end{example}
We adapt Li-Wu's expanded degenerations to torsion sheaves and consider the following extension of Example \ref{eg 1}.
\begin{example}(Relative $DT_{4}$/$DT_{3}$)\label{ex rel DT4/DT3}

Let $X_{+}=Y_{1}\times \mathbb{P}^{1}$ which contains $Y=(Y_{1}\times 0)\sqcup (Y_{1}\times \infty)$ as an anti-canonical divisor,
where $Y_{1}$ is a compact Calabi-Yau 3-fold.
%$H^{*}(X_{+},\mathbb{Z})\cong H_{2}(Q,\mathbb{Z})\oplus H_{2}(\mathbb{P}^{1},\mathbb{Z})\cong \mathbb{Z}\times\mathbb{Z}$.
We denote $\mathfrak{M}_{c}(Y_{1})$ to be a Gieseker moduli space of torsion-free semi-stable sheaves on $Y_{1}$ with
Chern character $c\in H^{even}(Y_{1},\mathbb{Q})$ (we assume there is no strictly semi-stable sheaf), and denote $\mathfrak{M}_{c}(X_{+})$ to
be the moduli space of sheaves on $X_{+}$ which are push-forward of stable sheaves in $\mathfrak{M}_{c}(Y_{1}\times t)$ for some $t$
$(\mathfrak{M}_{c}(X_{+})\cong \mathfrak{M}_{c}(Y_{1})\times \mathbb{P}^{1})$. Let $\iota: Y_{1}\times t\hookrightarrow X_{+}$
be the natural embedding. For any stable sheaf $\mathcal{F}\in \mathfrak{M}_{c}(Y_{1})$, as in Lemma 6.4 of \cite{caoleung}, we have
\begin{equation}Ext^{1}_{X_{+}}(\iota_{*}\mathcal{F},\iota_{*}\mathcal{F})\cong Ext^{1}_{Y_{1}}(\mathcal{F},\mathcal{F})\oplus
Ext^{0}_{Y_{1}}(\mathcal{F},\mathcal{F}), \nonumber \end{equation}
\begin{equation}\label{max iso}Ext^{2}_{X_{+}}(\iota_{*}\mathcal{F},\iota_{*}\mathcal{F})\cong Ext^{2}_{Y_{1}}(\mathcal{F},\mathcal{F})
\oplus Ext^{2}_{Y_{1}}(\mathcal{F},\mathcal{F})^{*}. \end{equation}
Furthermore, under the above identifications, a Kuranishi map
\begin{equation}\kappa_{\iota_{*}\mathcal{F}}: Ext^{1}_{X_{+}}(\iota_{*}\mathcal{F},\iota_{*}\mathcal{F})\rightarrow
Ext^{2}_{X_{+}}(\iota_{*}\mathcal{F},\iota_{*}\mathcal{F})\nonumber \end{equation}
satisfies
\begin{equation} \kappa_{\iota_{*}\mathcal{F}}(a,b)=(\kappa_{\mathcal{F}}(a),0),  \nonumber \end{equation}
for some Kuranishi map $\kappa_{\mathcal{F}}: Ext^{1}_{Y_{1}}(\mathcal{F},\mathcal{F})\rightarrow Ext^{2}_{Y_{1}}(\mathcal{F},\mathcal{F})$ of
$\mathfrak{M}_{c}(Y_{1})$ at $\mathcal{F}$.

To have a well-defined restriction map, we introduce $X_{+}[1]_{0}=\Delta_{-1}\cup X_{+}\cup\Delta_{1}$,
where $\Delta_{\pm1}\cong Y_{1}\times\mathbb{P}^{1}$, and consider $\mathfrak{M}_{c}(X_{+}[1]_{0})$ to
be the union of $\mathfrak{M}_{c}(Y_{1})\times\mathbb{C}^{*}$ with the $\mathbb{C}^{*}$-equivalence classes of
$(\mathfrak{M}_{c}(Y_{1}\times 0)\times\mathbb{C}^{*})\sqcup (\mathfrak{M}_{c}(Y_{1}\times \infty)\times\mathbb{C}^{*})$, i.e.
\begin{equation}\mathfrak{M}_{c}(X_{+}[1]_{0})\cong \mathfrak{M}_{c}(Y_{1})\times\mathbb{P}^{1}.  \nonumber \end{equation}
As the supports are disjoint, we have $\mathcal{T}or^{i\geq1}(\iota_{*}\mathcal{F},\mathcal{O}_{Y})=0$ which implies
a long exact sequence similar to the one in Lemma \ref{fund exact seq for dege }, and get a canonical isomorphism
$Ext^{2}_{X_{+}[1]_{0}}(\iota_{*}\mathcal{F},\iota_{*}\mathcal{F})\cong Ext^{2}_{X_{+}[1]_{0}}(\iota_{*}\mathcal{F},\iota_{*}\mathcal{F})^{*}$
as (\ref{(-2)-stru}). In this specific case, there exists a canonical maximal isotropic subspace $Ext^{2}_{Y_{1}}(\mathcal{F},\mathcal{F})
\subseteq Ext^{2}_{X_{+}[1]_{0}}(\iota_{*}\mathcal{F},\iota_{*}\mathcal{F})$ (\ref{max iso}) and $Ext^{2}_{Y_{1}}(\mathcal{F},\mathcal{F})$'s
form a sheaf over $\mathfrak{M}_{c}(X_{+}[1]_{0})$, thus Conjectures \ref{conj0}, \ref{conj1} hold. Analogs to Theorem 6.5 \cite{caoleung}, the relative $DT_{4}$ virtual cycle satisfies
\begin{equation}[\mathfrak{M}^{rel}_{c}(X_{+}[1]_{0})]^{vir}=DT_{3}(\mathfrak{M}_{c}(Y_{1}))\cdot[\mathbb{P}^{1}]\in
H_{2}(\mathfrak{M}_{c}(X_{+}[1]_{0}),\mathbb{Z}),\nonumber \end{equation}
where $DT_{3}(\mathfrak{M}_{c}(Y_{1}))$ is the Donaldson-Thomas invariant defined by Thomas \cite{th}.
\end{example}

\section{Appendix on relative $DT_{4}$ invariants for ideal sheaves and gluing formulas}

\subsection{Li-Wu's good degeneration of Hilbert schemes}
In this subsection, we recall some basic notions and facts of Li-Wu's good degeneration of Hilbert schemes.
The precise definitions are left to their papers \cite{liwu}, \cite{wu}. \\
${}$ \\
\textbf{The stack of expanded degenerations}.
We first introduce the stack of expanded degenerations for pairs $(X_{\pm},Y)$. We replace a pair $(X_{+},Y)$ by expanded pairs of length $n$,
$(X_{+}[n]_{0},Y[n]_{0})$, i.e.
\begin{equation}X_{+}[n]_{0}=X_{+}\cup\Delta_{1}\cup\cdot\cdot\cdot\cup\Delta_{n},  \nonumber \end{equation}
which is a chain of smooth irreducible components intersecting transversally with $\Delta_{i}$ to be the $i^{th}$ copy of
$\Delta\triangleq \mathbb{P}(\mathcal{N}_{Y/X_{+}}\oplus \mathcal{O}_{Y})$. $\Delta$ is a $\mathbb{P}^{1}$ bundle over $Y$ with
two canonical divisors $Y_{\pm}$
such that $\mathcal{N}_{Y_{+}/X_{+}}\cong\mathcal{N}_{Y/X_{+}}$ and $\mathcal{N}_{Y_{-}/X_{+}}\cong\mathcal{N}_{Y/X_{+}}^{*}$. We denote
$Y[n]_{0}=Y_{+}$ to be the divisor in the last component $\Delta_{n}$. In fact, we can consider families of expanded pairs,
$(X_{+}[n],Y[n])$ over affine space $\mathbb{A}^{n}$ such that over $0\in \mathbb{A}^{n}$ it coincides with $(X_{+}[n]_{0},Y[n]_{0})$.
Then there exists a pair of Artin stacks $(\mathfrak{X}_{+},\mathfrak{Y})$ as the direct limit of stack quotients of $(X_{+}[n],Y[n])$
by certain group actions. The projection of $(X_{+}[n],Y[n])$ to the affine space $\mathbb{A}^{n}$ induces a morphism
$\mathfrak{Y}\subseteq \mathfrak{X}_{+}\rightarrow\mathfrak{A}_{\diamond}$, where $\mathfrak{A}_{\diamond}$ is the direct limit of
some stack quotients of the affine space $\mathbb{A}^{n+1}$.

To formula the gluing formula, we also need to replace the family $\mathcal{X}\rightarrow C$ by its expanded degeneration
$\mathfrak{X}\rightarrow \mathfrak{C}$, where $\mathfrak{X}$ is the direct limit of stack quotients of $X[n]$ and $X[n]$ is a
family over $C[n]\triangleq C\times_{\mathbb{A}^{1}}\mathbb{A}^{n+1}$, $\mathfrak{C}\triangleq C\times_{\mathbb{A}^{1}}\mathfrak{A}$
and $\mathfrak{A}$ is another
stack quotient of the affine space $\mathbb{A}^{n+1}$. A smooth chart of $\mathfrak{X}_{0}\triangleq \mathfrak{X}\times_{C}0$ is
\begin{equation}X[n]_{0}=X_{+}\cup\Delta_{1}\cup\cdot\cdot\cdot\cup\Delta_{n}\cup X_{-},  \nonumber \end{equation}
which is a chain of smooth irreducible components intersecting transversally with $\Delta_{i}$ to be the $i^{th}$ copy of
$\Delta\triangleq \mathbb{P}(\mathcal{N}_{Y/X_{+}}\oplus \mathcal{O}_{Y})\cong\mathbb{P}(\mathcal{N}_{Y/X_{-}}\oplus \mathcal{O}_{Y})$. We also denote
$\Delta_{0}=X_{+}$, $\Delta_{n+1}=X_{-}$.

If we consider $X[n]_{0}$ as the gluing of $(X_{\pm}[n]_{0},Y[n]_{0})$, we need to specify one of its divisor in some $\Delta_{i}$. This is
called a node-marking and there exists an Artin stack $\mathfrak{X}_{0}^{\dag}$ which is the collection of families in $\mathfrak{X}_{0}$ with
node-markings. One can construct a stack $\mathfrak{C}_{0}^{\dag}$ and an arrow $\mathfrak{C}_{0}^{\dag}\rightarrow \mathfrak{C}$ that fits
into a Cartesian product
\begin{equation}
\xymatrix{
  \mathfrak{X}_{0}^{\dag} \ar[d]_{} \ar[r]^{ }
                & \mathfrak{X} \ar[d]^{ }  \\
  \mathfrak{C}_{0}^{\dag}   \ar[r]_{ }
                & \mathfrak{C}        }
\nonumber \end{equation}
By Proposition 2.13 \cite{liwu}, there exists a canonical isomorphism $\mathfrak{C}_{0}^{\dag}\cong \mathfrak{A}_{\diamond}$.

To fix Hilbert polynomials of ideal sheaves over $X[n]_{0}$ and decompose them into ideal sheaves of fixed Hilbert polynomials on
$(X_{\pm}[n]_{0},Y[n]_{0})$, we introduce
\begin{equation}\Lambda_{P}^{spl}\triangleq \{\delta=(\delta_{\pm},\delta_{0}) \textrm{ }  |  \textrm{ } \delta_{+}+\delta_{-}-\delta_{0}=P \},
\nonumber \end{equation}
where $\delta_{\pm}$, $\delta_{0}$, $P$ are polynomials in $\mathcal{A}\triangleq\mathcal{A}^{*}\cup\{0\}$, and $\mathcal{A}^{*}$ is the set of
$\mathbb{Q}$-coefficient polynomials whose leading terms are of the form $a_{r}\frac{k^{r}}{r!}$ with $a_{r}\in \mathbb{Z}_{+}$.

We define the stack $\mathfrak{X}_{0}^{\dag,\delta}$ whose closed points are
$(X[n]_{0},Y_{k},w)$, where $w$ is a function such that
\begin{equation}w(\Delta_{[0,k-1]})=\delta_{-}, \quad w(\Delta_{[k,n+1]})=\delta_{+} , \quad w(Y_{k})=\delta_{0} \nonumber \end{equation}

We similarly assign functions $w_{\pm}$ to $(X_{\pm}[n]_{0},Y[n]_{0})$ with
\begin{equation}w_{\pm}(\Delta_{[0,n]})=\delta_{\pm}, \quad w_{\pm}(Y[n]_{0})=\delta_{0} \nonumber \end{equation}
and define stacks $\mathfrak{X}_{\pm}^{\delta_{\pm},\delta_{0}}$. Then there exists stacks $\mathfrak{A}_{\diamond}^{\delta_{\pm},\delta_{0}}$ so
that we have Cartesian product
\begin{equation}
\xymatrix{
  \mathfrak{X}_{\pm}^{\delta_{\pm},\delta_{0}} \ar[d]_{} \ar[r]^{ }
                & \mathfrak{X}_{\pm} \ar[d]^{ }  \\
   \mathfrak{A}_{\diamond}^{\delta_{\pm},\delta_{0}}   \ar[r]_{ }
                & \mathfrak{A}_{\diamond}        }
\nonumber \end{equation}
By gluing two components, we obtain the following commutative diagram
\begin{equation}
\xymatrix{
   \mathfrak{X}_{+}^{\delta_{\pm},\delta_{0}}\times \mathfrak{X}_{-}^{\delta_{\pm},\delta_{0}}  \ar[d]_{} \ar[r]^{ }
                & \mathfrak{X}_{0}^{\dag,\delta} \ar[d]^{ }  \\
    \mathfrak{A}_{\diamond}^{\delta_{+},\delta_{0}}\times \mathfrak{A}_{\diamond}^{\delta_{-},\delta_{0}}  \ar[r]^{\quad \quad \cong}
                &  \mathfrak{C}_{0}^{\dag,\delta}  }. \nonumber \end{equation}
We denote $\mathfrak{C}_{0}^{\dag,P}=\bigsqcup_{\delta\in\Lambda_{P}^{spl}}\mathfrak{C}_{0}^{\dag,\delta}$ and then
have a natural morphism $\Phi_{\delta}:\mathfrak{C}_{0}^{\dag,\delta}\rightarrow \mathfrak{C}^{P}$ as the composition of the imbedding
$\mathfrak{C}_{0}^{\dag,\delta}\rightarrow \mathfrak{C}_{0}^{\dag,P}$ with forgetful map $\mathfrak{C}_{0}^{\dag,P}\rightarrow \mathfrak{C}^{P}$.
\begin{lemma}(Li-Wu, Proposition 2.19 \cite{liwu})\label{complete int1} ${}$ \\
There are canonical line bundles with sections $(L_{\delta},s_{\delta})$ on $\mathfrak{C}^{P}$ indexed by $\delta\in\Lambda_{P}^{spl}$, such that

(1) let $t$ be the standard coordinate function on $\mathbb{A}^{1}$ and $\pi:\mathfrak{C}^{P}\rightarrow \mathbb{A}^{1}$ be
the tautological projection, then
\begin{equation}\bigotimes_{\delta\in\Lambda_{P}^{spl}}L_{\delta}\cong\mathcal{O}_{\mathfrak{C}^{P}}, \quad
\prod_{\delta\in\Lambda_{P}^{spl}}s_{\delta}=\pi^{*}t; \nonumber \end{equation}

(2) $\Phi_{\delta}$ factors through $s_{\delta}^{-1}(0)\subseteq \mathfrak{C}^{P}$ and there exists an isomorphism
$s_{\delta}^{-1}(0)\cong \mathfrak{C}_{0}^{\dag,\delta}$.
\end{lemma}
This states that $\mathfrak{C}^{P}_{0}\subseteq \mathfrak{C}^{P}$ is a complete intersection substack with
$\bigsqcup_{\delta\in\Lambda_{P}^{spl}}\mathfrak{C}^{\dag,\delta}_{0}$ as its normalization. \\
${}$ \\
\textbf{Moduli stacks of stable ideal sheaves}.
%We now consider relative ideal sheaves on $\mathfrak{Y}\subseteq \mathfrak{X}_{+}$.
By Theorem 4.14 \cite{liwu}, there exists a Deligne-Mumford stack $\mathfrak{I}^{P}_{\mathfrak{X}/\mathfrak{C}}$ which is finite type,
separated and proper over $C$. It is a good degeneration of Hilbert scheme of subschemes of $X/C$ with fixed Hilbert polynomial $P$ in the sense that
\begin{equation}\mathfrak{I}^{P}_{\mathfrak{X}/\mathfrak{C}}\times_{C}t\cong Hilb^{P}(X_{t}), \quad t\neq0   \nonumber \end{equation}
and the central fiber $\mathfrak{I}^{P}_{\mathfrak{X}_{0}/\mathfrak{C}_{0}}\triangleq \mathfrak{I}^{P}_{\mathfrak{X}/\mathfrak{C}}\times_{C}0$
has a good obstruction theory. We recall that any closed point of $\mathfrak{I}^{P}_{\mathfrak{X}_{0}/\mathfrak{C}_{0}}$ is an ideal sheaf
$I_{Z}$ in some $X[n]_{0}$ such that

(1) $\mathcal{O}_{Z}$ is normal to all $Y_{i}\subseteq X[n]_{0}$, i.e. $Tor^{1}_{\mathcal{O}_{X[n]_{0}}}(\mathcal{O}_{Z},\mathcal{O}_{Y_{i}})=0$;

(2) $Aut_{\mathfrak{X}}(I_{Z})$ is finite.

(3) The Hilbert polynomial of $\mathcal{O}_{Z}$ is $P$. \\

We define
\begin{equation}\mathfrak{I}^{\delta}_{\mathfrak{X}_{0}^{\dag}/\mathfrak{C}_{0}^{\dag}}\triangleq
\mathfrak{I}^{P}_{\mathfrak{X}/\mathfrak{C}}\times_{\mathfrak{C}^{P}}\mathfrak{C}_{0}^{\dag,\delta}. \nonumber \end{equation}
It parameterizes ideal sheaves $I_{Z}$'s on $X[n]_{0}$ with a node-marking $Y_{k}\subseteq X[n]_{0}$ so that the Hilbert polynomials
of $\mathcal{O}_{Z}$ restricted to $\cup_{i<k}\Delta_{i}$, to $\cup_{i\geq k}\Delta_{i}$ and to $Y_{k}$ are
$\delta_{-}$, $\delta_{+}$ and $\delta_{0}$ respectively.

We can similarly define the moduli stack of stable relative ideal sheaves for $\mathfrak{Y}\subseteq \mathfrak{X}_{\pm}$ with pair Hilbert polynomial
$(\delta_{\pm},\delta_{0})$, denoted by $\mathfrak{I}^{\delta_{\pm},\delta_{0}}_{\mathfrak{X}_{\pm}/\mathfrak{A}_{\diamond}}$, which are
finite type, separated and proper Deligne-Mumford stacks (Theorem 4.15 \cite{liwu}).
The relations between $\mathfrak{I}^{P}_{\mathfrak{X}/\mathfrak{C}}$, $\mathfrak{I}^{\delta}_{\mathfrak{X}_{0}^{\dag}/\mathfrak{C}_{0}^{\dag}}$
and $\mathfrak{I}^{\delta_{\pm},\delta_{0}}_{\mathfrak{X}_{\pm}/\mathfrak{A}_{\diamond}}$ are described as follows.
\begin{lemma}(Li-Wu, Theorem 5.27 \cite{liwu})\label{lemma on decomposition} ${}$ \\
(1) There exists natural restriction morphisms
$\mathfrak{I}^{\delta_{\pm},\delta_{0}}_{\mathfrak{X}_{\pm}/\mathfrak{A}_{\diamond}}\rightarrow Hilb^{\delta_{0}}_{Y}$,
where $Hilb^{\delta_{0}}_{Y}$ is the Hilbert scheme on $Y$ with fixed Hilbert polynomial $\delta_{0}$, and an isomorphism
\begin{equation}\mathfrak{I}^{\delta_{-},\delta_{0}}_{\mathfrak{X}_{-}/\mathfrak{A}_{\diamond}}\times_{Hilb^{\delta_{0}}_{Y}}
\mathfrak{I}^{\delta_{+},\delta_{0}}_{\mathfrak{X}_{+}/\mathfrak{A}_{\diamond}}\rightarrow
\mathfrak{I}^{\delta}_{\mathfrak{X}_{0}^{\dag}/\mathfrak{C}_{0}^{\dag}}. \nonumber \end{equation}
${}$ \\
(2) Let $(L_{\delta},s_{\delta})$ be as in Lemma \ref{complete int1} and $\pi_{P}: \mathfrak{I}^{P}_{\mathfrak{X}/\mathfrak{C}}\rightarrow
\mathfrak{C}^{P}$ be the natural projection. Then
\begin{equation}\bigotimes_{\delta\in\Lambda_{P}^{spl}}\pi_{P}^{*}L_{\delta}\cong\mathcal{O}_{\mathfrak{I}^{P}_{\mathfrak{X}/\mathfrak{C}}}, \quad
\prod_{\delta\in\Lambda_{P}^{spl}}\pi_{P}^{*}s_{\delta}=\pi_{P}^{*}\pi^{*}t; \nonumber \end{equation}
As closed substacks of $\mathfrak{I}^{P}_{\mathfrak{X}/\mathfrak{C}}$, we have $\mathfrak{I}^{\delta}_{\mathfrak{X}_{0}^{\dag}/\mathfrak{C}_{0}^{\dag}}
\cong (\pi_{P}^{*}s_{\delta}=0)$.
\end{lemma}

\subsection{Relative $DT_{4}$ virtual cycles}
We study obstruction theories of Deligne-Mumford stacks $\mathfrak{I}^{\delta}_{\mathfrak{X}_{0}^{\dag}/\mathfrak{C}_{0}^{\dag}}$
and $\mathfrak{I}^{\delta_{\pm},\delta_{0}}_{\mathfrak{X}_{\pm}/\mathfrak{A}_{\diamond}}$.
\begin{lemma}\label{fund exact seq for dege }
We take a closed point $[I_{Z}]\in \mathfrak{I}^{\delta_{+},\delta_{0}}_{\mathfrak{X}_{+}/\mathfrak{A}_{\diamond}}$ with
$Z\subseteq X_{+}[n]_{0}$, then for $Y=Y[n]_{0}$,
we have a short exact sequence
\begin{equation}\label{SES in X[n]}0\rightarrow I_{Z}\otimes \mathcal{O}_{X_{+}[n]_{0}}(-Y)\rightarrow I_{Z}\rightarrow
I_{Z}\otimes\mathcal{O}_{Y}\rightarrow 0
\end{equation}
and a long exact sequence
\begin{equation}\cdot\cdot\cdot\rightarrow Ext^{1}(I_{Z},I_{Z}\otimes \mathcal{O}_{X_{+}[n]_{0}}(-Y))\rightarrow
Ext^{1}(I_{Z},I_{Z})\rightarrow Ext^{1}_{Y}(I_{Z}\otimes\mathcal{O}_{Y},I_{Z}\otimes\mathcal{O}_{Y})
\rightarrow \nonumber \end{equation}
\begin{equation}\rightarrow Ext^{2}(I_{Z},I_{Z}\otimes \mathcal{O}_{X_{+}[n]_{0}}(-Y))
\rightarrow  Ext^{2}(I_{Z},I_{Z})\rightarrow Ext^{2}_{Y}(I_{Z}\otimes\mathcal{O}_{Y},I_{Z}\otimes\mathcal{O}_{Y})
\rightarrow \nonumber \end{equation}
\begin{equation}\rightarrow Ext^{3}(I_{Z},I_{Z}\otimes \mathcal{O}_{X_{+}[n]_{0}}(-Y))\rightarrow
Ext^{3}(I_{Z},I_{Z})\rightarrow \cdot\cdot\cdot. \quad\quad\quad\quad\quad\quad \quad\quad\quad\quad \quad \nonumber \end{equation}
\end{lemma}
\begin{proof}
We tensor $0\rightarrow I_{Z}\rightarrow \mathcal{O}_{X_{+}[n]_{0}}\rightarrow \mathcal{O}_{Z}\rightarrow 0$ with $\mathcal{O}_{Y}$ and get
\begin{equation}\label{tor vanishi0}\mathcal{T}or_{\mathcal{O}_{X_{+}[n]_{0}}}^{i+1}(\mathcal{O}_{Z},\mathcal{O}_{Y})\cong
\mathcal{T}or_{\mathcal{O}_{X_{+}[n]_{0}}}^{i}(I_{Z},\mathcal{O}_{Y}), \quad i\geq1.  \end{equation}
Applying tensor product with $\mathcal{O}_{Z}$ to
$0\rightarrow\mathcal{O}_{X_{+}[n]_{0}}(-Y)\rightarrow\mathcal{O}_{X_{+}[n]_{0}}\rightarrow\mathcal{O}_{Y}\rightarrow0$, we get
\begin{equation}\label{tor vanishi}\mathcal{T}or_{\mathcal{O}_{X_{+}[n]_{0}}}^{i\geq2}(\mathcal{O}_{Z},\mathcal{O}_{Y})=0. \end{equation}
These ensure that we have the short exact sequence (\ref{SES in X[n]}) after tensoring  $0\rightarrow\mathcal{O}_{X_{+}[n]_{0}}(-Y)
\rightarrow\mathcal{O}_{X_{+}[n]_{0}}\rightarrow\mathcal{O}_{Y}\rightarrow0$ with $I_{Z}$.
We then take $Hom(I_{Z},\cdot)$ to (\ref{SES in X[n]}) and are left to show $Ext^{*}(I_{Z},I_{Z}\otimes\mathcal{O}_{Y})\cong
Ext^{*}_{Y}(I_{Z}\otimes\mathcal{O}_{Y},I_{Z}\otimes\mathcal{O}_{Y})$. We have a spectral sequence
\begin{equation}H^{*}(X_{+}[n]_{0},\mathcal{E}xt^{*}(I_{Z},I_{Z}\otimes\mathcal{O}_{Y}))\Rightarrow Ext^{*}(I_{Z},I_{Z}\otimes\mathcal{O}_{Y}).
\nonumber \end{equation}
By Corollary 2.9 \cite{wu}, we can take a finite length locally free resolution $E^{\bullet}\rightarrow I_{Z}\rightarrow0$. Then
\begin{equation}H^{*}(X_{+}[n]_{0},\mathcal{E}xt^{*}(I_{Z},I_{Z}\otimes\mathcal{O}_{Y}))\cong H^{*}(X_{+}[n]_{0},\mathcal{E}xt^{*}
(\mathcal{O}_{X_{+}[n]_{0}},\mathcal{O}_{Y})\otimes End(E^{\bullet})) \nonumber \end{equation}
\begin{equation}\cong H^{*}(X_{+}[n]_{0},End(E^{\bullet})\otimes\mathcal{O}_{Y})\cong H^{*}(X_{+}[n]_{0},\iota_{*}End(E^{\bullet}|_{Y}))\cong
Ext^{*}_{Y}(E^{\bullet}|_{Y},E^{\bullet}|_{Y}), \nonumber \end{equation}
where $\iota: Y\hookrightarrow X_{+}[n]_{0}$ is the closed imbedding.
By (\ref{tor vanishi0}), (\ref{tor vanishi}), we have $\mathcal{T}or_{\mathcal{O}_{X_{+}[n]_{0}}}^{i\geq1}(I_{Z},\mathcal{O}_{Y})=0$
which implies that $E^{\bullet}|_{Y}\rightarrow I_{Z}|_{Y}\rightarrow0$ is still
an resolution. Thus $Ext^{*}_{Y}(I_{Z}|_{Y},I_{Z}|_{Y})\cong Ext^{*}(I_{Z},I_{Z}\otimes\mathcal{O}_{Y})$.
\end{proof}
%\begin{remark}\label{remark on bridge of bdl w shf} %\end{remark}
The above long exact sequence is the ideal sheaf version of the long exact sequence in Lemma \ref{def-obs LES}.
We consider the following extension of virtual cycles for bundles (\textbf{Cases I-III}) to ideal sheaf cases.
%Definitions of relative $DT_{4}$ virtual cycles in \textbf{Cases II-III} for bundles extend directly to ideal sheaf cases.
\begin{definition}\label{rel DT4 vc for ideal sheaves}
Let $Y$ be an anti-canonical divisor of a complex projective 4-fold $X_{+}$, and \begin{equation}r:\mathfrak{I}^{\delta_{+},\delta_{0}}_{\mathfrak{X}_{+}/\mathfrak{A}_{\diamond}}\rightarrow Hilb^{\delta_{0}}_{Y}
\nonumber \end{equation}
be Li-Wu's restriction morphism in Lemma \ref{lemma on decomposition}. We assume $\mathfrak{I}^{\delta_{+},\delta_{0}}_{\mathfrak{X}_{+}/\mathfrak{A}_{\diamond}}$ is a smooth moduli scheme (all Kuranishi maps vanish).

The relative $DT_{4}$ virtual cycle for $\mathfrak{I}^{\delta_{+},\delta_{0}}_{\mathfrak{X}_{+}/\mathfrak{A}_{\diamond}}$ is its usual fundamental class provided that any one of the following conditions holds, \\
(1) $r$ is surjective between smooth moduli spaces, and the obstruction bundle $Ob(\mathfrak{I}^{\delta_{+},\delta_{0}}_{\mathfrak{X}_{+}/\mathfrak{A}_{\diamond}})=0$, \\
(2) $r$ is injective between smooth moduli spaces (at least when restricted to a neighbourhood $U(r(\mathfrak{I}^{\delta_{+},\delta_{0}}_{\mathfrak{X}_{+}/\mathfrak{A}_{\diamond}}))$ of
$r(\mathfrak{I}^{\delta_{+},\delta_{0}}_{\mathfrak{X}_{+}/\mathfrak{A}_{\diamond}})$ in $Hilb^{\delta_{0}}_{Y}$), and $rk(Ob(\mathfrak{I}^{\delta_{+},\delta_{0}}_{\mathfrak{X}_{+}/\mathfrak{A}_{\diamond}}))=
codim(\mathfrak{I}^{\delta_{+},\delta_{0}}_{\mathfrak{X}_{+}/\mathfrak{A}_{\diamond}},
U(r(\mathfrak{I}^{\delta_{+},\delta_{0}}_{\mathfrak{X}_{+}/\mathfrak{A}_{\diamond}})))$.
\end{definition}
${}$ \\
\textbf{A modification by twisting $K_{X_{\pm}}^{\frac{1}{2}}$}. In general, extensions of virtual cycles for bundles (\textbf{Cases I-III}) to ideal sheaves are not straightforward. We study the extension for \textbf{Case I}.
\begin{proposition}\label{self dual obs}
We take a smooth Calabi-Yau 3-fold $Y$ in complex projective 4-folds $X_{\pm}$ as their anti-canonical divisors.
We assume any $I_{C}\in Hilb^{\delta_{0}}(Y)$ satisfies $Ext^{1}(I_{C},I_{C})=0$. Then for any closed point of
$\mathfrak{I}^{\delta_{\pm},\delta_{0}}_{\mathfrak{X}_{\pm}/\mathfrak{A}_{\diamond}}$, say $[I_{Z_{\pm}}]$ with $Z_{\pm}\subseteq X_{\pm}[n]_{0}$,
we have canonical isomorphisms
\begin{equation}\label{(-2)-stru}Ext^{i}_{X_{\pm}[n]_{0}}(I_{Z_{\pm}},I_{Z_{\pm}})_{0}\cong Ext^{4-i}_{X_{\pm}[n]_{0}}(I_{Z_{\pm}},I_{Z_{\pm}})^{*}_{0},
%\textrm{ } Ext^{2}_{X_{\pm}[n]_{0}}(I_{Z_{\pm}},I_{Z_{\pm}})\cong Ext^{2}_{X_{\pm}[n]_{0}}(I_{Z_{\pm}},I_{Z_{\pm}})^{*}.
\textrm{ } i=1,2.
 \end{equation}
Furthermore, under the isomorphism in Lemma \ref{lemma on decomposition},
\begin{equation}\mathfrak{I}^{\delta}_{\mathfrak{X}_{0}^{\dag}/\mathfrak{C}_{0}^{\dag}}\cong\mathfrak{I}^{\delta_{-},\delta_{0}}_
{\mathfrak{X}_{-}/\mathfrak{A}_{\diamond}}\times_{Hilb^{\delta_{0}}_{Y}} \mathfrak{I}^{\delta_{+},\delta_{0}}_{\mathfrak{X}_{+}/\mathfrak{A}_{\diamond}},
\nonumber \end{equation}
where a closed point of $\mathfrak{I}^{\delta}_{\mathfrak{X}_{0}^{\dag}/\mathfrak{C}_{0}^{\dag}}$ is written as $I_{Z}=I_{Z_{+}}\cup I_{Z_{-}}$,
with $Z\subseteq X[n_{+}+n_{-}]_{0}$, $Z_{\pm}\subseteq X_{\pm}[n_{\pm}]_{0}$,
we have canonical isomorphisms of trace-free extension groups
\begin{equation}Ext^{*}_{X[n_{+}+n_{-}]_{0}}(I_{Z},I_{Z})_{0}\cong Ext^{*}_{X_{+}[n_{+}]_{0}}(I_{Z_{+}},I_{Z_{+}})_{0}
\oplus Ext^{*}_{X_{-}[n_{-}]_{0}}(I_{Z_{-}},I_{Z_{-}})_{0}, \nonumber \end{equation}
under which the non-degenerate quadratic forms on $Ext^{2}_{X[n_{+}+n_{-}]_{0}}(I_{Z},I_{Z})_{0}$ and
$Ext^{2}_{X_{\pm}[n]_{0}}(I_{Z_{\pm}},I_{Z_{\pm}})$ (\ref{(-2)-stru}) are preserved.
\end{proposition}
\begin{proof}
We apply the trace-free version of Lemma \ref{fund exact seq for dege } to the case when $Y$ is the Cartier divisor
associated with the dualizing sheaf of $X_{\pm}[n]_{0}$ and get canonical isomorphisms by Serre duality.

We take a closed point $I_{Z}\in \mathfrak{I}^{\delta}_{\mathfrak{X}_{0}^{\dag}/\mathfrak{C}_{0}^{\dag}}$ with
$Z\subseteq X[n_{+}+n_{-}]_{0}=X_{+}[n_{+}]_{0}\cup_{Y} X_{-}[n_{-}]_{0}$ and restrict to get $I_{Z_{\pm}}\subseteq \mathcal{O}_{X_{\pm}[n_{\pm}]}$.
As \cite{mpt}, \cite{liwu}, we then get an exact triangle
%\begin{equation}RHom_{X[n_{+}+n_{-}]_{0}}(I_{Z},I_{Z})\rightarrow RHom_{X_{+}[n_{+}]_{0}}(I_{Z}|_{X[n_{+}]_{0}},I_{Z}|_{X[n_{+}]_{0}})\bigoplus RHom_{X_{-}[n_{-}]_{0}}(I_{Z}|_{X[n_{-}]_{0}},I_{Z}|_{X[n_{-}]_{0}})    \nonumber \end{equation}
\begin{equation}RHom_{X[n_{+}+n_{-}]_{0}}(I_{Z},I_{Z})_{0}\rightarrow \bigoplus
RHom_{X_{\pm}[n_{\pm}]_{0}}(I_{Z}|_{X_{\pm}[n_{\pm}]_{0}},I_{Z}|_{X_{\pm}[n_{\pm}]_{0}})_{0}\rightarrow
RHom_{Y}(I_{Z}|_{Y},I_{Z}|_{Y})_{0}.   \nonumber \end{equation}
We take cohomology, use the assumption to deduce $H^{*}(RHom_{Y}(I_{Z}|_{Y},I_{Z}|_{Y})_{0})=0$.
\end{proof}
From Proposition \ref{self dual obs}, one may expect to get a
$(-2)$-shifted symplectic structure on $\mathfrak{I}^{\delta_{\pm},\delta_{0}}_{\mathfrak{X}_{\pm}/\mathfrak{A}_{\diamond}}$
if any $I_{C}\in Hilb^{\delta_{0}}(Y)$ satisfies $Ext^{1}(I_{C},I_{C})=0$. Then as Borisov-Joyce did in \cite{bj},
one expects to use BBJ's type local Darboux charts, partition of unity and homotopical algebra to obtain
D-orbifolds associated with $\mathfrak{I}^{\delta_{\pm},\delta_{0}}_{\mathfrak{X}_{\pm}/\mathfrak{A}_{\diamond}}$ and
$\mathfrak{I}^{\delta}_{\mathfrak{X}_{0}^{\dag}/\mathfrak{C}_{0}^{\dag}}$.
In particular, analogs to Theorem \ref{DT4 main thm}, there should exist homology classes
$[\mathfrak{I}^{\delta_{\pm},\delta_{0}}_{\mathfrak{X}_{\pm}/\mathfrak{A}_{\diamond}}]^{vir}\in
H_{*}(\mathfrak{I}^{\delta_{\pm},\delta_{0}}_{\mathfrak{X}_{\pm}/\mathfrak{A}_{\diamond}},\mathbb{Q})$,
$[\mathfrak{I}^{\delta}_{\mathfrak{X}_{0}^{\dag}/\mathfrak{C}_{0}^{\dag}}]^{vir}\in
H_{*}(\mathfrak{I}^{\delta}_{\mathfrak{X}_{0}^{\dag}/\mathfrak{C}_{0}^{\dag}},\mathbb{Q})$ if the associated D-orbifolds of
$\mathfrak{I}^{\delta_{\pm},\delta_{0}}_{\mathfrak{X}_{\pm}/\mathfrak{A}_{\diamond}}$ and $\mathfrak{I}^{\delta}_{\mathfrak{X}_{0}^{\dag}/
\mathfrak{C}_{0}^{\dag}}$ are orientable \cite{joyce1}.

However, under the assumption: any $I_{C}\in Hilb^{\delta_{0}}(Y)$ satisfies $Ext^{1}(I_{C},I_{C})=0$,
the non-degenerate quadratic forms on $Ext^{2}_{X_{\pm}[n]_{0}}(I_{Z_{\pm}},I_{Z_{\pm}})$'s (see Proposition \ref{self dual obs})
don't have to glue together, and the D-orbifold associated with $\mathfrak{I}^{\delta_{\pm},\delta_{0}}_{\mathfrak{X}_{\pm}/\mathfrak{A}_{\diamond}}$
might not exist (see Example \ref{ideal sheaves of one pt}). For gluing, we discuss the case when there exist a square root $K_{X_{\pm}}^{\frac{1}{2}}$
of $K_{X_{\pm}}$ with $K_{X_{\pm}}^{\frac{1}{2}}\otimes K_{X_{\pm}}^{\frac{1}{2}}\cong K_{X_{\pm}}$.
\begin{lemma}\label{new obs space}
Let $X_{+}$ be a complex projective 4-fold with a square root $K_{X_{+}}^{\frac{1}{2}}$, $Y_{i}$ ($i=1,2$)
be two smooth zero loci of sections of $K_{X_{+}}^{-\frac{1}{2}}$ with $K_{Y_{i}}=0$
($\Leftrightarrow \mathcal{N}_{Y_{i}/X_{+}}\cong\mathcal{O}_{Y_{i}}$) and $Y_{1}\cap Y_{2}=\emptyset$. We take $Y=Y_{1}\sqcup Y_{2}$
which a smooth anti-canonical divisor of $X_{+}$. Then for any closed point
$[I_{Z}]\in \mathfrak{I}^{\delta_{+},\delta_{0}}_{\mathfrak{X}_{+}/\mathfrak{A}_{\diamond}}$ with $Z\subseteq X_{+}[n]_{0}$ and
$Y[n]_{0}=Y_{1}\sqcup Y_{2}$, we have a short exact sequence
\begin{equation} 0\rightarrow I_{Z}\otimes K_{X_{+}[n]_{0}}^{\frac{1}{2}}\rightarrow I_{Z}\rightarrow I_{Z}\otimes\mathcal{O}_{Y_{1}}\rightarrow 0
\nonumber \end{equation}
and a long exact sequence
\begin{equation}0\rightarrow Ext^{1}_{X_{+}[n]_{0}}(I_{Z},I_{Z}\otimes K_{X_{+}[n]_{0}}^{\frac{1}{2}})_{0}\rightarrow
Ext^{1}_{X_{+}[n]_{0}}(I_{Z},I_{Z})_{0}\rightarrow Ext^{1}_{Y_{1}}(I_{Z}\otimes\mathcal{O}_{Y_{1}},I_{Z}\otimes\mathcal{O}_{Y_{1}})_{0}
\rightarrow \nonumber \end{equation}
\begin{equation}\rightarrow Ext^{2}_{X_{+}[n]_{0}}(I_{Z},I_{Z}\otimes K_{X_{+}[n]_{0}}^{\frac{1}{2}})_{0}\rightarrow
Ext^{2}_{X_{+}[n]_{0}}(I_{Z},I_{Z})_{0}\rightarrow Ext^{2}_{Y_{1}}(I_{Z}\otimes\mathcal{O}_{Y_{1}},I_{Z}\otimes\mathcal{O}_{Y_{1}})_{0} \rightarrow\cdot\cdot\cdot,
\nonumber \end{equation} %Ext^{3}(I_{Z},I_{Z}\otimes K_{X_{+}[n]_{0}}^{\frac{1}{2}})\rightarrow
%Ext^{3}(I_{Z},I_{Z})\rightarrow \cdot\cdot\cdot.
where $K_{X_{+}[n]_{0}}^{\frac{1}{2}}$ is a square root of the dualizing sheaf of $X_{+}[n]_{0}$.

Furthermore, if any $I_{C}\in Hilb^{\delta_{0}}(Y)$ satisfies $Ext^{1}_{Y}(I_{C},I_{C})=0$, we have canonical isomorphisms
\begin{equation}\phi_{1}: Ext^{2}_{X_{+}[n]_{0}}(I_{Z},I_{Z}\otimes K_{X_{+}[n]_{0}})_{0}\cong Ext^{2}_{X_{+}[n]_{0}}(I_{Z},I_{Z}
\otimes K_{X_{+}[n]_{0}}^{\frac{1}{2}})_{0},
\nonumber \end{equation}
\begin{equation}\phi_{2}: Ext^{2}_{X_{+}[n]_{0}}(I_{Z},I_{Z}\otimes K_{X_{+}[n]_{0}}^{\frac{1}{2}})_{0}\cong Ext^{2}_{X_{+}[n]_{0}}(I_{Z},I_{Z})_{0},
\nonumber \end{equation}
\end{lemma}
\begin{proof}
It is similar to the proof of Lemma \ref{fund exact seq for dege }. The isomorphism $\phi_{1}$ is derived by tensoring
$0\rightarrow I_{Z}\otimes K_{X_{+}[n]_{0}}^{\frac{1}{2}}\rightarrow I_{Z}\rightarrow I_{Z}\otimes\mathcal{O}_{Y_{1}}\rightarrow 0$
with $K_{X_{+}[n]_{0}}^{\frac{1}{2}}$ and taking the long exact sequence.
\end{proof}
The reason of introducing $Ext^{2}_{X_{+}[n]_{0}}(I_{Z},I_{Z}\otimes K_{X_{+}[n]_{0}}^{\frac{1}{2}})_{0}$ is that
Serre duality pairing defines a natural non-degenerate quadratic form on it.
If $Ext^{2}_{X_{+}[n]_{0}}(I_{Z_{+}},I_{Z_{+}}\otimes K_{X_{+}[n]_{0}}^{\frac{1}{2}})_{0}$'s are glued to be a
sheaf over the moduli space, the Serre duality pairing will probably 'glue'. By Proposition \ref{self dual obs},
there is a non-degenerate quadratic form on $Ext^{2}_{X_{+}[n]_{0}}(I_{Z},I_{Z})_{0}$. We make a comparison between them.
\begin{proposition}\label{comparison of quad form}
Let $X_{+}$ be a complex projective 4-fold with a square root $K_{X_{+}}^{\frac{1}{2}}$, $Y_{i}$ ($i=1,2$) be
two smooth zero loci of sections of $K_{X_{+}}^{-\frac{1}{2}}$ with $K_{Y_{i}}=0$ ($\Leftrightarrow \mathcal{N}_{Y_{i}/X_{+}}\cong\mathcal{O}_{Y_{i}}$)
and $Y_{1}\cap Y_{2}=\emptyset$. We take $Y=Y_{1}\sqcup Y_{2}$ which a smooth anti-canonical divisor of $X_{+}$ and assume any
$I_{C}\in Hilb^{\delta_{0}}(Y)$ satisfies $Ext^{1}_{Y}(I_{C},I_{C})=0$. Then for any closed point
$[I_{Z}]\in \mathfrak{I}^{\delta_{+},\delta_{0}}_{\mathfrak{X}_{+}/\mathfrak{A}_{\diamond}}$ with $Z\subseteq X_{+}[n]_{0}$ and
$Y[n]_{0}=Y_{1}\sqcup Y_{2}$,
we have a commutative diagram
\begin{equation}\xymatrix{
 Ext^{2}_{X_{+}[n]_{0}}(I_{Z},I_{Z}\otimes K_{X_{+}[n]_{0}})_{0} \ar[d]_{\phi_{1}}^{\cong} \ar[dr]^{\phi_{3}}_{\cong}       \\
  Ext^{2}_{X_{+}[n]_{0}}(I_{Z},I_{Z}\otimes K_{X_{+}[n]_{0}}^{\frac{1}{2}})_{0}\ar[r]_{ \quad  \quad  \phi_{2}}^{\quad  \quad  \cong}
  &  Ext^{2}_{X_{+}[n]_{0}}(I_{Z},I_{Z})_{0},}
\nonumber \end{equation}
where $\phi_{1}$, $\phi_{2}$ are defined in Lemma \ref{new obs space} and $\phi_{3}$ is the isomorphism induced from the long exact sequence
in Lemma \ref{fund exact seq for dege }. Furthermore, $\phi_{2}$ is an isometry with respect to the Serre duality pairing on
$Ext^{2}_{X_{+}[n]_{0}}(I_{Z},I_{Z}\otimes K_{X_{+}[n]_{0}}^{\frac{1}{2}})_{0}$ and the quadratic form on $Ext^{2}_{X_{+}[n]_{0}}(I_{Z},I_{Z})_{0}$
defined in Proposition \ref{self dual obs}.
\end{proposition}
\begin{proof}
The commutativity is because isomorphisms $\phi_{1}$, $\phi_{2}$ are pairings with sections in $H^{0}(X_{+}[n]_{0},K_{X_{+}[n]_{0}}^{-\frac{1}{2}})$
corresponding to Cartier divisors $Y_{1}$, $Y_{2}$, and $\phi_{3}$ is the pairing with a section in $H^{0}(X_{+}[n]_{0},K_{X_{+}[n]_{0}}^{-1})$
corresponding to $Y=Y_{1}\sqcup Y_{2}$. Then it is easy to check $\phi_{2}$ is an isometry.
\end{proof}
%By Proposition \ref{bbj model of Hilb}, we can choose a local Kuranishi map at $I_{Z_{+}}\in Hilb^{\delta_{+},\delta_{0},st}_{X_{+}[n]_{0}/Y_{+}[n]_{0}}$,
%\begin{equation}\kappa: Ext^{1}_{X_{+}[n]_{0}}(I_{Z_{+}},I_{Z_{+}})_{0}\rightarrow Ext^{2}_{X_{+}[n]_{0}}(I_{Z_{+}},I_{Z_{+}})_{0}
%\nonumber \end{equation}
%such that $Q(\kappa,\kappa)=0$, where $Q$ is the quadratic form in Proposition \ref{self dual obs}.
%
%If $K_{X_{+}}^{\frac{1}{2}}$ exists with a smooth $CY_{3}$ as the zero loci of a section of $K_{X_{+}}^{-\frac{1}{2}}$, by Lemma \ref{comparison of quad form}, we can rewrite a local Kuranishi map at $I_{Z_{+}}\in Hilb^{\delta_{+},\delta_{0},st}_{X_{+}[n]_{0}/Y_{+}[n]_{0}}$ to be
%\begin{equation}\tilde{\kappa}: Ext^{1}_{X_{+}[n]_{0}}(I_{Z_{+}},I_{Z_{+}})_{0}\rightarrow Ext^{2}_{X_{+}[n]_{0}}(I_{Z_{+}},I_{Z_{+}}\otimes K_{X_{+}[n]_{0}}^{\frac{1}{2}})_{0} \nonumber \end{equation}
%such that $\widetilde{Q}(\tilde{\kappa},\tilde{\kappa})=0$, where $\widetilde{Q}$ is the Serre duality pairing on $Ext^{2}_{X_{+}[n]_{0}}(I_{Z_{+}},I_{Z_{+}}\otimes K_{X_{+}[n]_{0}}^{\frac{1}{2}})_{0}$. Then once $Ext^{2}_{X_{+}[n]_{0}}(I_{Z_{+}},I_{Z_{+}}\otimes K_{X_{+}[n]_{0}}^{\frac{1}{2}})_{0}$'s glue as a sheaf over the moduli space, the Serre duality pairing will also glue.
\begin{conjecture}\label{conj0}
Let $Y$ be a smooth Calabi-Yau 3-fold in complex projective 4-folds $X_{\pm}$ as their anti-canonical divisors.
We assume any $I_{C}\in Hilb^{\delta_{0}}(Y)$ satisfies $Ext^{1}(I_{C},I_{C})=0$. Then there exists a D-orbifold
associated with Deligne-Mumford stack $\mathfrak{I}^{\delta}_{\mathfrak{X}_{0}^{\dag}/\mathfrak{C}_{0}^{\dag}}$, i.e.
they have the same underlying topological structures.

Furthermore, if canonical bundles of $X_{\pm}$ admit square roots $K_{X_{\pm}}^{\frac{1}{2}}$, and there exist $Y_{i}$ ($i=1,2$)
which are smooth zero loci of sections of $K_{X_{\pm}}^{-\frac{1}{2}}$ with $K_{Y_{i}}=0$
($\Leftrightarrow \mathcal{N}_{Y_{i}/X_{\pm}}\cong\mathcal{O}_{Y_{i}}$) such that $Y=Y_{1}\cup Y_{2}$, $Y_{1}\cap Y_{2}=\emptyset$, then there exist D-orbifolds associated with Deligne-Mumford stacks $\mathfrak{I}^{\delta_{\pm},\delta_{0}}_{\mathfrak{X}_{\pm}/\mathfrak{A}_{\diamond}}$.

\end{conjecture}
\begin{conjecture}\label{conj1}
We take a smooth Calabi-Yau 3-fold $Y$ in complex projective 4-folds $X_{\pm}$ as their anti-canonical divisors.
We assume any $I_{C}\in Hilb^{\delta_{0}}(Y)$ satisfies $Ext^{1}(I_{C},I_{C})=0$ (we assume $Hilb^{\delta_{0}}(Y)$
consists of one point without loss of generality), then we have
\begin{equation}[\mathfrak{I}^{\delta_{\pm},\delta_{0}}_{\mathfrak{X}_{\pm}/\mathfrak{A}_{\diamond}}]^{vir}\in
H_{*}(\mathfrak{I}^{\delta_{\pm},\delta_{0}}_{\mathfrak{X}_{\pm}/\mathfrak{A}_{\diamond}},\mathbb{Q}), \textrm{ }
[\mathfrak{I}^{\delta}_{\mathfrak{X}_{0}^{\dag}/\mathfrak{C}_{0}^{\dag}}]^{vir}\in
H_{*}(\mathfrak{I}^{\delta}_{\mathfrak{X}_{0}^{\dag}/\mathfrak{C}_{0}^{\dag}},\mathbb{Q}) \nonumber \end{equation}
if there exist orientable D-orbifolds \cite{joyce1} associated with
$\mathfrak{I}^{\delta_{\pm},\delta_{0}}_{\mathfrak{X}_{\pm}/\mathfrak{A}_{\diamond}}$ and
$\mathfrak{I}^{\delta}_{\mathfrak{X}_{0}^{\dag}/\mathfrak{C}_{0}^{\dag}}$.
Furthermore, under the isomorphism
\begin{equation}\mathfrak{I}^{\delta_{-},\delta_{0}}_{\mathfrak{X}_{-}/\mathfrak{A}_{\diamond}}\times_{Hilb^{\delta_{0}}_{Y}}
\mathfrak{I}^{\delta_{+},\delta_{0}}_{\mathfrak{X}_{+}/\mathfrak{A}_{\diamond}}
\cong \mathfrak{I}^{\delta}_{\mathfrak{X}_{0}^{\dag}/\mathfrak{C}_{0}^{\dag}} \nonumber \end{equation}
in Lemma \ref{lemma on decomposition}, we have an identification of virtual cycles
\begin{equation}[\mathfrak{I}^{\delta}_{\mathfrak{X}_{0}^{\dag}/\mathfrak{C}_{0}^{\dag}}]^{vir}=
[\mathfrak{I}^{\delta_{+},\delta_{0}}_{\mathfrak{X}_{+}/\mathfrak{A}_{\diamond}}]^{vir}\times
[\mathfrak{I}^{\delta_{-},\delta_{0}}_{\mathfrak{X}_{-}/\mathfrak{A}_{\diamond}}]^{vir}, \nonumber \end{equation}
if we choose appropriate orientations on D-orbifolds associated with
$\mathfrak{I}^{\delta_{\pm},\delta_{0}}_{\mathfrak{X}_{\pm}/\mathfrak{A}_{\diamond}}$ and
$\mathfrak{I}^{\delta}_{\mathfrak{X}_{0}^{\dag}/\mathfrak{C}_{0}^{\dag}}$.
\end{conjecture}

\subsection{A conjectural gluing formula}
We state a conjectural gluing formula of $DT_{4}$ invariants for a simple degeneration $\mathcal{X}\rightarrow C$ of projective $CY_{4}$'s.
We assume $\omega_{\mathcal{X}/C}=0$ and $X_{0}=X_{+}\cup_{Y}X_{-}$ with $Y$ as an anti-canonical divisor of $X_{\pm}$.

We first consider the virtual cycle of $\mathfrak{I}^{P}_{\mathfrak{X}/\mathfrak{C}}$. Because of the triviality of the relative canonical bundle
$\omega_{\mathcal{X}/C}=0$, %it's easy to check that $\mathfrak{I}^{P}_{\mathfrak{X}/\mathfrak{C}}$ has a $(-2)$-shifted symplectic structure.
as Conjecture \ref{conj0}, \ref{conj1}, there should exist a D-orbifold associated with $\mathfrak{I}^{P}_{\mathfrak{X}/\mathfrak{C}}$ and
a Borel-Moore homology class $[\mathfrak{I}^{P}_{\mathfrak{X}/\mathfrak{C}}]^{vir}\in
H_{*}^{BM}(\mathfrak{I}^{P}_{\mathfrak{X}/\mathfrak{C}},\mathbb{Q})$ if the D-orbifold is orientable \cite{joyce1}.

${}$ \\
\textbf{The comparison of $[\mathfrak{I}^{P}_{\mathfrak{X}/\mathfrak{C}}]^{vir}$ and $[\mathfrak{I}^{\delta}_
{\mathfrak{X}_{0}^{\dag}/\mathfrak{C}_{0}^{\dag}}]^{vir}$}.
By Lemma \ref{lemma on decomposition},
$\mathfrak{I}^{\delta}_{\mathfrak{X}_{0}^{\dag}/\mathfrak{C}_{0}^{\dag}}$ is the zero loci of a section $\pi_{P}^{*}s_{\delta}$ of
a complex line bundle $\pi_{P}^{*}L_{\delta}$ on $\mathfrak{I}^{P}_{\mathfrak{X}/\mathfrak{C}}$. Meanwhile, by \cite{liwu},
the obstruction theory of $\mathfrak{I}^{\delta}_{\mathfrak{X}_{0}^{\dag}/\mathfrak{C}_{0}^{\dag}}$ is the pull-back of the obstruction theory of
$\mathfrak{I}^{P}_{\mathfrak{X}/\mathfrak{C}}$, thus we should have
\begin{equation}\label{rel obs theory compatible}[\mathfrak{I}^{\delta}_{\mathfrak{X}_{0}^{\dag}/\mathfrak{C}_{0}^{\dag}}]^{vir}=
c_{1}(\pi_{P}^{*}L_{\delta},\pi_{P}^{*}s_{\delta})\cap [\mathfrak{I}^{P}_{\mathfrak{X}/\mathfrak{C}}]^{vir}, \end{equation}
where
$c_{1}(\pi_{P}^{*}L_{\delta},\pi_{P}^{*}s_{\delta})\in H^{2}(\mathfrak{I}^{P}_{\mathfrak{X}/\mathfrak{C}},\mathfrak{I}^{P}_{\mathfrak{X}/\mathfrak{C}}-
\mathfrak{I}^{\delta}_{\mathfrak{X}_{0}^{\dag}/\mathfrak{C}_{0}^{\dag}})$ is the localized first Chern class (Proposition 19.1.2 \cite{fulton}) and
$c_{1}(\pi_{P}^{*}L_{\delta},\pi_{P}^{*}s_{\delta}) \cap  : H_{*}^{BM}(\mathfrak{I}^{P}_{\mathfrak{X}/\mathfrak{C}})\rightarrow
H_{*-2}(\mathfrak{I}^{\delta}_{\mathfrak{X}_{0}^{\dag}/\mathfrak{C}_{0}^{\dag}})$ is the cap product \cite{iverson}.

Summing over all splitting $\delta$ of $P$, we get
\begin{equation}\label{central fiber equal}
\sum_{\delta\in\Lambda_{P}^{spl}}[\mathfrak{I}^{\delta}_{\mathfrak{X}_{0}^{\dag}/\mathfrak{C}_{0}^{\dag}}]^{vir}
=c_{1}(\bigotimes_{\delta\in\Lambda_{P}^{spl}}\pi_{P}^{*}L_{\delta},\prod_{\delta\in\Lambda_{P}^{spl}}\pi_{P}^{*}s_{\delta})\cap
[\mathfrak{I}^{P}_{\mathfrak{X}/\mathfrak{C}}]^{vir}
%=c_{1}(\mathcal{O}_{\mathfrak{I}^{P}_{\mathfrak{X}/\mathfrak{C}}},\pi_{P}^{*}\pi^{*}t)\cap [\mathfrak{I}^{P}_{\mathfrak{X}/\mathfrak{C}}]^{vir}.
\end{equation}
%\begin{remark}
%The above equality (\ref{rel obs theory compatible}) should be true if we set up the
%virtual theory for derived schemes with $(-2)$-shifted symplectic structures appropriately \cite{joyce}.
%One can easily check it in the special cases where we define $DT_{4}$ virtual cycles before in \cite{caoleung}.
%\end{remark}
${}$ \\
\textbf{The comparison of $[\mathfrak{I}^{P}_{\mathfrak{X}/\mathfrak{C}}]^{vir}$ and $[\mathfrak{I}^{P}_{X_{t}}]^{vir}$}.
For $t\neq0$, we have $\mathfrak{I}^{P}_{\mathfrak{X}/\mathfrak{C}}\times_{C}t\cong \mathfrak{I}^{P}_{X_{t}}$. The obstruction theory of
$\mathfrak{I}^{P}_{X_{t}}$ is the pull back of the obstruction theory of $\mathfrak{I}^{P}_{\mathfrak{X}/\mathfrak{C}}$. Without loss of
generality, we assume $C=\mathbb{A}^{1}$ and similarly obtain
\begin{equation}\label{generic fiber equal}
[\mathfrak{I}^{P}_{X_{t}}]^{vir}=c_{1}(\mathcal{O}_{\mathfrak{I}^{P}_{\mathfrak{X}/\mathfrak{C}}},\pi_{P}^{*}\pi^{*}t)\cap
[\mathfrak{I}^{P}_{\mathfrak{X}/\mathfrak{C}}]^{vir},  \end{equation}
where $[\mathfrak{I}^{P}_{X_{t}}]^{vir}$ is the $DT_{4}$ virtual cycle mentioned in Theorem \ref{DT4 main thm}.
\begin{conjecture}\label{conj2}
We take a simple degeneration $\mathcal{X}\rightarrow C$ of projective $CY_{4}$'s with $\omega_{\mathcal{X}/C}=0$ such that $X_{0}=X_{+}\cup_{Y}X_{-}$
and $Y$ is an anti-canonical divisor of $X_{\pm}$. Then there exists a D-orbifold associated with $\mathfrak{I}^{P}_{\mathfrak{X}/\mathfrak{C}}$,
and a Borel-Moore homology class
\begin{equation}[\mathfrak{I}^{P}_{\mathfrak{X}/\mathfrak{C}}]^{vir}\in H_{*}^{BM}(\mathfrak{I}^{P}_{\mathfrak{X}/\mathfrak{C}},\mathbb{Q})
\nonumber \end{equation}
if the corresponding D-orbifold is orientable.

Furthermore, equalities (\ref{rel obs theory compatible}), (\ref{generic fiber equal}) hold if we choose appropriate orientations for
D-orbifolds associated with $\mathfrak{I}^{P}_{\mathfrak{X}/\mathfrak{C}}$, $\mathfrak{I}^{\delta}_{\mathfrak{X}_{0}^{\dag}/\mathfrak{C}_{0}^{\dag}}$
and $\mathfrak{I}^{P}_{X_{t}}$.
\end{conjecture}
To introduce the gluing formula, we make the following definition.
\begin{definition}\label{def of family DT4}
Let $\mathcal{X}\rightarrow C$ be a simple degeneration of projective $CY_{4}$'s such that $X_{0}=X_{+}\cup_{Y}X_{-}$ with $Y$ as
an anti-canonical divisor of $X_{\pm}$. $P$ is a polynomial and $\mathfrak{I}^{P}_{\mathfrak{X}/\mathfrak{C}}\rightarrow C$ is
Li-Wu's good degeneration of $Hilb^{P}(X_{t})$, $t\neq0$. We assume for any $\delta=(\delta_{\pm},\delta_{0})\in\Lambda_{P}^{spl}$,
any closed point $I_{C}\in Hilb^{\delta_{0}}(Y)$ satisfies $Ext^{1}_{Y}(I_{C},I_{C})=0$. Then the family version $DT_{4}$ invariant
of $Hilb^{P}(X_{t})$, $t\neq0$ is a map
\begin{equation}DT_{4}(\mathfrak{I}^{P}_{X_{t}}): Sym^{*}\big(H_{*}(\mathcal{X},\mathbb{Z}) \otimes \mathbb{Z}[x_{1},x_{2},...]\big)
\rightarrow \mathbb{Z} \nonumber \end{equation}
such that
\begin{equation}DT_{4}(\mathfrak{I}^{P}_{X_{t}})((\gamma_{1},P_{1}),(\gamma_{2},P_{2}),...)=\int_{[\mathfrak{I}^{P}_{X_{t}}]^{vir}}
\mu(\gamma_{1},P_{1})\cup \mu(\gamma_{2},P_{2})\cup...  , \nonumber \end{equation}
where $\gamma_{i}\in H_{*}(X,\mathbb{Z})$, $\mu(,)$ is the $\mu$-map defined in Definition \ref{mu map relative}
with respect to the universal sheaf of $\mathfrak{I}^{P}_{\mathfrak{X}/\mathfrak{C}}$ and we view
$\mathfrak{I}^{P}_{X_{t}}\hookrightarrow \mathfrak{I}^{P}_{\mathfrak{X}/\mathfrak{C}}$ as a closed substack to do integration. \\

We take a K\"{u}nneth type decomposition of the cohomology class
\begin{equation}(\mu(\gamma_{1},P_{1})\cup \mu(\gamma_{2},P_{2})\cup...)|_{\mathfrak{I}^{\delta}_{\mathfrak{X}_{0}^{\dag}/\mathfrak{C}_{0}^{\dag}}}
=\sum_{i}\tau_{+,\delta,i}\boxtimes \tau_{-,\delta,i}\in H^{*}(\mathfrak{I}^{\delta_{-},\delta_{0}}_{\mathfrak{X}_{-}/\mathfrak{A}_{\diamond}}\times
\mathfrak{I}^{\delta_{+},\delta_{0}}_{\mathfrak{X}_{+}/\mathfrak{A}_{\diamond}}). \nonumber \end{equation}
Then the relative $DT_{4}$ invariant of $\mathfrak{I}^{\delta_{\pm},\delta_{0}}_{\mathfrak{X}_{\pm}/\mathfrak{A}_{\diamond}}$ with respect to
$\tau_{\pm,\delta,i}$ is
\begin{equation}DT_{4}(\mathfrak{I}^{\delta_{\pm},\delta_{0}}_{\mathfrak{X}_{\pm}/\mathfrak{A}_{\diamond}})(\tau_{\pm,\delta,i})
=\int_{[\mathfrak{I}^{\delta_{\pm},\delta_{0}}_{\mathfrak{X}_{\pm}/\mathfrak{A}_{\diamond}}]^{vir}}\tau_{\pm,\delta,i}\in \mathbb{Q}.
\nonumber \end{equation}
\end{definition}
We state a gluing formula of $DT_{4}$ invariants on Calabi-Yau 4-folds based on previous conjectures.
\begin{theorem}\label{dege formula}
Let $\mathcal{X}\rightarrow C$ be a simple degeneration of projective $CY_{4}$ such that $\omega_{\mathcal{X}/C}=0$ and $X_{0}=X_{+}\cup_{Y}X_{-}$ with $Y$
as an anti-canonical divisor of $X_{\pm}$. $P$ is a polynomial and $\mathfrak{I}^{P}_{\mathfrak{X}/\mathfrak{C}}\rightarrow C$ is
Li-Wu's good degeneration of $Hilb^{P}(X_{t})$, $t\neq0$. We assume for any $\delta=(\delta_{\pm},\delta_{0})\in\Lambda_{P}^{spl}$,
any closed point $I_{C}\in Hilb^{\delta_{0}}(Y)$ satisfies $Ext^{1}_{Y}(I_{C},I_{C})=0$ (without loss of generality,
we assume $Hilb^{\delta_{0}}(Y)$ consists of one point for simplicity), then for $t\neq0\in C$,
\begin{equation}DT_{4}(\mathfrak{I}^{P}_{X_{t}})((\gamma_{1},P_{1}),(\gamma_{2},P_{2}),...)=\sum_{\delta\in\Lambda_{P}^{spl},i}
DT_{4}(\mathfrak{I}^{\delta_{+},\delta_{0}}_{\mathfrak{X}_{+}/\mathfrak{A}_{\diamond}})(\tau_{+,\delta,i})\cdot
DT_{4}(\mathfrak{I}^{\delta_{-},\delta_{0}}_{\mathfrak{X}_{-}/\mathfrak{A}_{\diamond}})(\tau_{-,\delta,i} ), \nonumber \end{equation}
where $\gamma_{i}\in H_{*}(X,\mathbb{Z})$, $\mu(,)$ is the $\mu$-map defined in Definition \ref{mu map relative}
with respect to the universal sheaf of $\mathfrak{I}^{P}_{\mathfrak{X}/\mathfrak{C}}$, $(\tau_{\pm,\delta,i})$
is the factor in a K\"{u}nneth type decomposition as in Definition \ref{def of family DT4}.
\end{theorem}
\begin{proof}
By Lemma \ref{lemma on decomposition}, Conjecture \ref{conj1}, \ref{conj2} and (\ref{central fiber equal}).
\end{proof}
\begin{remark}
By Corollary 2.16 \cite{wu}, if $P$ is the Hilbert polynomial associated with structure sheaves of points. The condition which says for any $\delta=(\delta_{\pm},\delta_{0})\in\Lambda_{P}^{spl}$ and any closed point $I_{C}\in Hilb^{\delta_{0}}(Y)$, we have $Ext^{1}_{Y}(I_{C},I_{C})=0$
is satisfied.
\end{remark}

\section{Appendix on the orientability of relative $DT_{4}$ theory}
In this section, we give a coherent description of orientability issues involved in definitions of
relative $DT_{4}$ virtual cycles (in Section 3) and then give some partial verification for the existence of orientations.

We take a smooth (Calabi-Yau) 3-fold $Y$ in a complex projective 4-fold $X$ as its anti-canonical divisor, and denote $\mathfrak{M}_{X}$
to be a moduli space of stable bundles on $X$ with fixed Chern classes.
Assuming conditions in Theorem \ref{restriction thm} are satisfied, we obtain a morphism
\begin{equation}r: \mathfrak{M}_{X}\rightarrow \mathfrak{M}_{Y} \nonumber \end{equation}
to a Gieseker moduli space of stable sheaves on $Y$.
We denote the determinant line bundle of $\mathfrak{M}_{X}$ by $\mathcal{L}_{X}$  with $\mathcal{L}_{X}|_{E}\cong det(H^{odd}(X,EndE))
\otimes det(H^{even}(X,EndE)))^{-1}$ (similarly for $\mathcal{L}_{Y}\rightarrow\mathfrak{M}_{Y}$).
In this set-up, there exists a canonical isomorphism
\begin{equation}\alpha:(\mathcal{L}_{\mathcal{M}_{X}})^{\otimes2}\cong r^{*}\mathcal{L}_{\mathcal{M}_{Y}}\footnote{See for instance Lemma 4.2 of \cite{caoleung3}.}.   \nonumber \end{equation}
\begin{definition}\label{def of rel ori}
%Let $X$ be a smooth projective $4$-fold with a smooth anti-canonical divisor $Y\in|K^{-1}_{X}|$, and $r: \mathcal{M}_{X}\rightarrow\mathcal{M}_{Y}$
%be a well-defined restriction morphism between coarse moduli spaces of simple sheaves on $X$ and $Y$ with fixed Chern classes respectively.
A \emph{relative orientation} for morphism $r$ consists of a square root $(\mathcal{L}_{\mathcal{M}_{Y}}|_{\mathcal{M}^{red}_{Y}})^{\frac{1}{2}}$
of the determinant line bundle $\mathcal{L}_{\mathcal{M}_{Y}}|_{\mathcal{M}^{red}_{Y}}$ and an isomorphism
\begin{equation}\theta:\mathcal{L}_{\mathcal{M}_{X}}|_{\mathcal{M}^{red}_{X}} \cong
r^{*}(\mathcal{L}_{\mathcal{M}_{Y}}|_{\mathcal{M}^{red}_{Y}})^{\frac{1}{2}}  \nonumber \end{equation}
such that $\theta\otimes\theta\cong \alpha$ holds over $\mathcal{M}^{red}_{X}$ for the isomorphism $\alpha$.
\end{definition}
\begin{proposition}\label{prop on relative ori}
The restriction morphism has a relative orientation in \textbf{Cases I-III} individually is equivalent to the existence of
an orientation in each corresponding case, i.e. \\
(i) the D-manifold associated with $\mathfrak{M}_{X}$ is orientable in \textbf{Case I};  \\
(ii) the self-dual obstruction bundle is orientable in \textbf{Case II};  \\
(iii) the self-dual reduced obstruction bundle is orientable in \textbf{Case III}.
\end{proposition}
\begin{proof}
As $\mathfrak{M}_{Y}$ is smooth in all cases, $\mathcal{L}_{Y}$ has a canonical square root given by $det(T\mathfrak{M}_{Y})$.

In \textbf{Case I}, $\mathfrak{M}_{Y}$ consists of finite number of points. The existence of relative orientations is obviously equivalent
to the existence of orientations for the D-manifold associated with $\mathfrak{M}_{X}$ (see also Theorem \ref{general existence of ori}).

In \textbf{Case II}, $H^{*}(X,EndE)$'s and $H^{1}(Y,EndE|_{Y})$'s are locally constant. We abuse notations and use them also to denote the corresponding bundles.
%\begin{equation}\label{equality 1}det(H^{1}(X,EndE)-r^{*}H^{1}(Y,EndE|_{Y})+H^{3}(X,EndE))\cong det(H^{2}(X,EndE)). \end{equation}
By the short exact sequence
\begin{equation}0\rightarrow H^{3}(X,EndE)^{*}\rightarrow H^{1}(X,EndE)\rightarrow r^{*}H^{1}(Y,EndE|_{Y})\rightarrow 0 \nonumber \end{equation}
in \textbf{Case II}, the relative orientability is equivalent to the structure group of the obstruction bundle $H^{2}(X,EndE)$ can be reduced to
$SO(\bullet,\mathbb{C})$, i.e. the self-dual obstruction bundle is orientable.

In \textbf{Case III}, the argument is similar as in \textbf{Case II}.
%$H^{3}(X,EndE)=0$, and $c_{1}(\mathcal{L}_{X})=r^{*}c_{1}(\mathfrak{M}_{Y})$ is equivalent to
%\begin{equation}c_{1}(H^{2}(X,EndE))-\big(c_{1}(H^{1}(X,EndE))-r^{*}c_{1}(H^{1}(Y,EndE|_{Y}))\big)=0, \nonumber \end{equation}
%which is the same as  $c_{1}(Ob_{\mathfrak{M}_{X}}^{red})=0$.
\end{proof}
%\begin{remark}
%The above result can be easily extended to ideal sheaves cases.
%\end{remark}
We have the following partial verification of the existence of relative orientations.
\begin{theorem}(Weak relative orientability) \label{thm on ori} ${}$ \\
Let $Y$ be a smooth anti-canonical divisor in a projective $4$-fold $X$ with $Tor(H_{*}(X,\mathbb{Z}))=0$, $E\rightarrow X$ be a
complex vector bundle with structure group $SU(N)$, where $N\gg0$.
Let $\mathcal{M}_{X}$ be a coarse moduli scheme of simple holomorphic structures on $E$, which has a well-defined restriction morphism
\begin{equation}r: \mathcal{M}_{X}\rightarrow \mathcal{M}_{Y}, \nonumber \end{equation}
to a proper coarse moduli scheme of simple bundles on $Y$ with fixed Chern classes.

Then there exists a square root $(\mathcal{L}_{\mathcal{M}_{Y}}|_{\mathcal{M}^{red}_{Y}})^{\frac{1}{2}}$ of
$\mathcal{L}_{\mathcal{M}_{Y}}|_{\mathcal{M}^{red}_{Y}}$ such that
\begin{equation}c_{1}(\mathcal{L}_{\mathcal{M}_{X}}|_{\mathcal{M}^{red}_{X}})
=r^{*}c_{1}((\mathcal{L}_{\mathcal{M}_{Y}}|_{\mathcal{M}^{red}_{Y}})^{\frac{1}{2}}),
\nonumber \end{equation}
%\begin{equation}\mathcal{L}_{\mathcal{M}_{X}}\cong r^{*}(\mathcal{L}_{\mathcal{M}_{Y}}^{\frac{1}{2}}), \nonumber \end{equation}
where $\mathcal{L}_{\mathcal{M}_{X}}$ (resp. $\mathcal{L}_{\mathcal{M}_{Y}}$) is the determinant line bundle of $\mathcal{M}_{X}$
(resp. $\mathcal{M}_{Y}$).
\end{theorem}
\begin{proof}
See the proof of Theorem 4.1 \cite{caoleung3}.
\end{proof}
Another partial result is given as follows.
\begin{proposition}\label{prop on rel ori}
We assume $H^{1}(\mathcal{M}_{X},\mathbb{Z}_{2})=0$. Then relative orientations
for restriction morphism $r: \mathcal{M}_{X}\rightarrow\mathcal{M}_{Y}$ exist.
\end{proposition}
\begin{proof}
See Proposition 4.6 of \cite{caoleung3}.
\end{proof}


\begin{thebibliography}{99}
\bibitem{atiyah0} M. F. Atiyah, \textit{New invariants of 3- and 4-Dimensional Manifolds}, Proceedings of Symposia in Pure Mathematics. Vol 48 (1988).
\bibitem{atiyah1} M. F. Atiyah, \textit{Topological quantum field theory}, Publ. Math. Inst. Hautes Etudes Sci. 68 (1989), 175-186.
\bibitem{as4} M. F. Atiyah and I. M. Singer, \textit{The index of elliptic operators: IV}, Ann. of Math. 92 (1970), 119-138.
\bibitem{behrend} K. Behrend, \textit{Donaldson-Thomas type invariants via microlocal geometry}, Ann. of Math. 170 (2009), 1307-1338.
\bibitem{bbs} K. Behrend, J. Bryan and B. Szendr\"{o}i, \textit{Motivic degree zero Donaldson-Thomas invariants}, Invent. Math. 192 (2013), 111-160.
\bibitem{bf} K. Behrend and B. Fantechi, \textit{The intrinsic normal cone}, Invent. Math. 128 (1997), 45-88.
\bibitem{bj} D. Borisov and D. Joyce \textit{Virtual fundamental classes for moduli spaces of sheaves on Calabi-Yau four-folds}, arXiv:1504.00690, 2015.
\bibitem{bbdjs} C. Brav, V. Bussi, D. Dupont, D. Joyce and B. Szendroi, \textit{Symmetries and stabilization for sheaves of vanishing cycles},
arXiv:1211.3259v3, 2013.
\bibitem{bbj} C. Brav, V. Bussi and D. Joyce \textit{A 'Darboux theorem' for derived schemes with shifted symplectic structure}, arXiv:1305.6302, 2013.
\bibitem{bridgeland}T. Bridgeland, \textit{Stability conditions on triangulated categories}, Annals of Mathematics, 166 (2007), 317-345.
\bibitem{calaque} D. Calaque, \textit{Lagrangian structures on mapping stacks and semi-classical TFTs}, arXiv:1306.3235, 2013.
\bibitem{cao} Y. Cao, \textit{Donaldson-Thomas theory for Calabi-Yau four-folds}, MPhil thesis, arXiv:1309.4230, 2013.
\bibitem{caoleung} Y. Cao and N. C. Leung, \textit{Donaldson-Thomas theory for Calabi-Yau 4-folds}, arXiv:1407.7659, 2014.
\bibitem{caoleung3} Y. Cao and N. C. Leung, \textit{Orientability for gauge theories on Calabi-Yau manifolds}, arXiv:1502.01141, 2015.
\bibitem{coxkatz}D. A. Cox and S. Katz, \textit{Mirror Symmetry and Algebraic Geometry}, Mathematical Surveys and Monographs 68, AMS, 1999.
\bibitem{d1}S. K. Donaldson, \textit{The orientation of Yang-Mills moduli spaces and 4-manifold topology}, J. Diff. Geom. 26 (1987), no. 3, 397-428.
\bibitem{donaldson} S. K. Donaldson, \textit{Floer homology groups in Yang-Mills theory}, Cambridge Tracts in Mathematics 147 (2002).
\bibitem{dt} S. K. Donaldson and R. P. Thomas, \textit{Gauge theory in higher dimensions}, in The Geometric Universe (Oxford, 1996), Oxford Univ.
    Press,Oxford, 1998, 31-47.
\bibitem{eg} D. Edidin and W. Graham \textit{Charateristic classes and quadric bundles}, Duke Math. J. 78 (1995), no. 2, 277-299.
\bibitem{flenner}H. Flenner, \textit{Restrictions of semistable bundles on projective varieties}, Comment. Math. Helvetici 59 (1984) 635-650.
\bibitem{fooo} K. Fukaya, Y.-G. Oh, H. Ohta, and K. Ono, \textit{Lagrangian Intersection Floer Theory: Anomaly and Obstruction. part I},
AMS/IP Studies in Advanced Mathematics, International Press.
\bibitem{fulton} W. Fulton, \textit{Intersection theory}, Springer-Verlag, Berlin, 1984.
\bibitem{gv} R. Gopakumar and C. Vafa, \textit{M-Theory and Topological Strings--II}, arXiv:hep-th/9812127, 1998.
\bibitem{hofer} H. Hofer, \textit{Polyfolds and Fredholm Theory}, arXiv:1412.4255, 2014.
\bibitem{hosono}S. Hosono, M.-H. Saito, A. Takahashi, \textit{Relative Lefschetz Action and BPS State Counting},
Internat. Math. Res. Notices, (2001), No. 15, 783-816.
\bibitem{hua1} Z. Hua, \textit{Orientation data on moduli space of sheaves on Calabi-Yau threefold}, arXiv:1212.3790v4, 2015.
\bibitem{ht} D. Huybrechts and R. P. Thomas, \textit{Deformation-obstruction theory for complexes via Atiyah and Kodaira Spencer classes},
 Math. Ann. 346 (2010), 545-569.
\bibitem{iverson}B. Iverson, \textit{Cohomology of sheaves}, Springer-Verlag, New York, 1986.
\bibitem{joyce1}D. Joyce, \textit{D-manifolds and d-orbifolds: a theory of derived differential geometry}, book in preparation, 2012.
Preliminary version available at Joyce's homepage.
\bibitem{joyce} D. Joyce, A series of three talks given Miami, January 2014, homepage of D. Joyce.
\bibitem{js} D. Joyce and Y. N. Song, \textit{A theory of generalized Donaldson-Thomas invariants}, Memoirs of the AMS, arXiv:0810.5645, 2010.
\bibitem{kl}Y. H. Kiem and J. Li, \textit{Categorification of Donaldson-Thomas invariants via Perverse Sheaves}, arXiv:1212.6444v4, 2013.
\bibitem{ks} M. Kontsevich and Y. Soibelman, \textit{Stability structures, motivic Donaldson-Thomas invariants and cluster transformations},
arXiv:0811.2435, 2008.
\bibitem{km}P. Kronheimer and T. Mrowka, \textit{Monopoles and three-manifolds}, New Mathematical Monographs, 10, Cambridge University Press,
Cambridge, 2007.
\bibitem{leung2} N. C. Leung, \textit{Topological Quantum Field Theory for Calabi-Yau threefolds and $G_{2}$ manifolds}, Adv. Theor. Math. Phys. 6 (2002) 575-591.
\bibitem{li1} J. Li, \textit{Stable morphisms to singular schemes and relative stable morphisms}, J. Differential Geom., 57(3) (2001), 509-578.
\bibitem{li2} J. Li, \textit{A degeneration formula of GW-invariants}, J. Differential Geom., 60(2) (2002), 199-293.
\bibitem{lt1} J. Li and G. Tian, \textit{Virtual moduli cycles and Gromov-Witten invariants of algebraic varieties},
Jour. Amer. Math. Soc. 11 (1998), 119-174.
\bibitem{liwu} J. Li and B. Wu, \textit{Good degeneration of Quot-schemes and coherent systems}, arXiv:1110.0390, 2011.
\bibitem{lq}  W. P. Li and Z. Qin, \textit{Stable rank-2 bundles on Calabi-Yau manifolds}, Internat. J. Math. 14 (2003), 1097-1120.
\bibitem{lieb} M. Lieblich, \textit{Moduli of complexes on a proper morphism}, J. Algebraic Geom. 15 (2006), 175-206.
\bibitem{mnop} D. Maulik, N. Nekrasov, A. Okounkov, and R. Pandharipande, \textit{Gromov-Witten theory and Donaldson-Thomas theory I},
Compositio Math. 142 (2006) 1263-1285.
\bibitem{mnop2} D. Maulik, N. Nekrasov, A. Okounkov, and R. Pandharipande, \textit{Gromov-Witten theory and Donaldson-Thomas theory II},
Compositio Math. 142 (2006) 1286-1304.
\bibitem{moop}  D. Maulik, A. Oblomkov, A. Okounkov, and R. Pandharipande, \textit{Gromov-Witten/Donaldson-Thomas correspondence for toric 3-folds},
Invent. Math. 186, 435-479, 2011.
\bibitem{mpt} D. Maulik, R. Pandharipande, and R. P. Thomas, \textit{Curves on K3 surfaces and modular forms}, J. Topol., 3(4):937-996, 2010.
With an appendix by A. Pixton.
\bibitem{mccleary}J. McCleary, \textit{A user's guide to spectral sequences}, Mathematics Lecture Series, Vol. 12, Publish or Perish, Wilmington,
DE, 1985.
\bibitem{mukai} S. Mukai, \textit{Symplectic structure of the moduli space of sheaves on an abelian or K3 surface}, Invent. math. 77, 101-116 (1984).
\bibitem{no}  N. Nekrasov and A. Okounkov, \textit{Membranes and Sheaves}, arXiv:1404.2323, 2014.
\bibitem{pandpixton} R. Pandharipande and A. Pixton, \textit{Gromov-Witten/Pairs correspondence for the quintic 3-fold}, arXiv:1206.5490. 2012.
\bibitem{pt} R. Pandharipande and R. P. Thomas, \textit{Curve counting via stable pairs in the derived category}, Invent. Math. 178: 407-447, 2009.
\bibitem{ptvv}T. Pantev, B. T\"{o}en, M. Vaqui\'{e} and G. Vezzosi, \textit{Shifted Symplectic Structures},
Publ. Math. Inst. Hautes Etudes Sci, June 2013, Volume 117, Issue 1, 271-328.
\bibitem{st} P. Seidel and R.P. Thomas, \textit{Braid group actions on derived categories of coherent sheaves}, Duke Math. J. 108 (2001), 37-108.
\bibitem{th} R. P. Thomas, \textit{A holomorphic Casson invariant for Calabi-Yau 3-folds, and bundles on K3 fibrations},
J. Differential Geometry. 54 (2000), 367-438.
\bibitem{toda} Y. Toda, \textit{Curve counting theories via stable objects I: DT/PT correspondence}, J. Amer. Math. Soc. 23, 1119-1157, 2010.
\bibitem{UY} K. Uhlenbeck and S. T. Yau, \textit{On the existence of Hermitian-Yang-Mills connections in stable vector bundles},
Comm. Pure Appl. Math. 39 (1986) 257-293.
\bibitem{wu} B. Wu, \textit{The moduli stack of stable relative ideal sheaves}, arXiv:math/0701074v1, 2007.
\bibitem{yang} D. Yang, \textit{The polyfold--Kuranishi correspondence I: A choice-independent theory of Kuranishi structures}, arXiv: 1402.7008v2, 2014.
\bibitem{yau} S. T. Yau, \textit{On the Ricci curvature of a compact K\"{a}hler manifold and the complex Monge-Amp\`{e}re equation I},
Comm. Pure Appl. Math. 31 (1978) 339-411.
\end{thebibliography}
\end{document}